\documentclass[11pt]{article}
\usepackage[utf8]{inputenc}
\usepackage[T1]{fontenc} 
\usepackage[french, english]{babel}
\usepackage[margin=2.5cm]{geometry}
\usepackage[english]{babel}
\usepackage{hyperref,xcolor}
\usepackage{amsmath}
\usepackage{amsthm}
\usepackage{amssymb}
\usepackage{enumitem}
\usepackage{mathtools}
\usepackage{tikz-cd}
\usepackage{ucs}
\DeclareMathOperator\Supp{Supp}
\DeclareMathOperator\ssup{sup}
\DeclareMathOperator\iinf{inf}
\DeclareMathOperator\Int{int}

\newtheorem{thm}{Theorem}[section]
\newtheorem{theorem}{Theorem}[section]

\theoremstyle{definition}
\newtheorem{definition}[theorem]{Definition}

\newtheorem{cor}[thm]{Corollary}

\newtheorem{lem}[thm]{Lemma}

\newtheorem{prop}[thm]{Proposition}

\DeclarePairedDelimiter\floor{\lfloor}{\rfloor}

\title{On Ulam widths of finitely presented infinite simple groups.}
\author{James Hyde and Yash Lodha}

\date{\today}

\newcommand{\R}{\mathbf{R}}
\newcommand{\Z}{\mathbf{Z}}

\begin{document}

\maketitle
\begin{abstract}
A fundamental notion in group theory, which originates in an article of Ulam and von Neumann from $1947$ is \emph{uniform simplicity}. A group $G$ is said to be \emph{$n$-uniformly simple} for $n \in \mathbf{N}$ if for every $f,g\in G\setminus \{id\}$, there is a product of no more than $n$ conjugates of $g$ and $g^{-1}$
that equals $f$.
Then $G$ is \emph{uniformly simple} if it is \emph{$n$-uniformly simple} for some $n \in \mathbf{N}$,
and we refer to the smallest such $n$ as the \emph{Ulam width}, denoted as $\mathcal{R}(G)$. 
If $G$ is simple but not uniformly simple, one declares $\mathcal{R}(G)=\infty$. 
In this article, we construct for each $n\in \mathbf{N}$, a finitely presented infinite simple group $G$ such that $n<\mathcal{R}(G)<\infty$. These are the first such examples among the class of finitely presented infinite simple groups. 
For the class of finitely generated (but not finitely presentable) infinite simple groups, the existence of such examples was settled in the work of Muranov \cite{Muranov}. However, this had remained open for the class of finitely presented infinite simple groups.
Our examples are also of type $F_{\infty}$, which means that they are fundamental groups of aspherical CW complexes with finitely many cells in each dimension.
Uniformly simple groups are in particular \emph{uniformly perfect}: there is an $n\in \mathbf{N}$ such that every element of the group can be expressed as a product of at most $n$ commutators of elements in the group.
We also show that the analogous notion of width for uniform perfection is unbounded for our family of finitely presented infinite simple groups.
To our knowledge, this is also the first such family.

\end{abstract}
\section{Introduction.}

The notion of a simple group is fundamental in group theory. 
A stronger notion that emerges in this setting is uniform simplicity, 
which goes back to an article of Ulam and von Neumann \cite{UlamVN}, where it was demonstrated that the identity component of the group of homeomorphisms of the circle or the $2$-sphere is a simple group. 
In his book \cite{Ulam}, Ulam explains ``For every $f$
and $\phi$ non-trivial and isotopic to identity homeomorphisms of the circle or the $2$-sphere,
there exists a fixed number $N$ of conjugates of $f$ or $f^{-1}$ whose product is $\phi$". This number was shown to be at most $23$ and Ulam raised the question of finding the optimal bound.
These notions were also discussed in the Scottish book \cite{Scottish} in connection with the work of Nunnally and Ulam.
Other authors have denoted uniformly simple groups as \emph{boundedly simple} or \emph{groups with finite covering number} \cite{Gismatullin}, \cite{GordeevSaxl} and \cite{EllersGordeevHerzog}. 
Not all simple groups are uniform simple. For instance, take the infinite alternating group.
On the other end of the spectrum, constructions of Osin and Hull-Osin provide finitely generated infinite groups which have Ulam width one \cite{Osin,HullOsin,HullOsinC}.

Nunnally proved that various groups of homeomorphisms are $3$-uniformly simple \cite{Nunnally}, and Tsuboi showed that the identity component of the group of $C^r$-diffeomorphisms of a compact connected $n$-manifold with handle decomposition without handles of index $\frac{n}{2}$ is $(16n+28)$-uniformly simple \cite{Tsuboi}. An elegant general criterion for uniform simplicity was provided in \cite{GalGis}, where it was demonstrated that many groups of dynamical origin are $n$-uniformly simple where $n$ is either $6$ or $9$, depending on the type of example. Uniform simplicity was established for various other groups of interest such as certain groups of interval exchange transformations and certain families of generalized Thompson's groups (with the upper bounds for the Ulam width being less than $25$ in all cases considered) in \cite{GuelmanLiousse}. We remark that these results provide upper bounds for the Ulam width, and it is in general a hard problem to determine the precise width. 

It follows from the work of Muranov that for each $n\in \mathbf{N}$, there exists a finitely generated (but not finitely presentable) infinite uniformly simple group $G$ such that $n<\mathcal{R}(G)<\infty$ (see \cite{Muranov}). 
Not only did this remain open for the class of finitely presented infinite simple groups, but among such groups that appear in the literature, $\mathcal{R}(G)$ (whenever determined) is either infinity or bounded above by small constants (see for example, \cite{GuelmanLiousse} and \cite{GalGis}).
In this article, we answer this in the affirmative:

\begin{thm}
For each $n\in \mathbf{N}$, there is a finitely presented (and type $F_{\infty}$) infinite simple group $G$ such that $n<\mathcal{R}(G)<\infty$.
In particular, $\mathcal{R}(G)$ is unbounded for the class of finitely presented infinite uniformly simple groups.
\end{thm}

To prove this, we provide the following construction of a new class of finitely presented infinite simple groups:

\begin{definition}\label{ourgroup}

For $x,y\in\mathbf{R}$, we define: 
$$(x,y)_{\mathbf{Z}}:=
\begin{cases}
|(x,y)\cap \mathbf{Z}| & \text{ if } x\leq y.\\
-|(y,x)\cap \mathbf{Z}| &\text{ if }y<x.
\end{cases}$$

Here $|.|$ denotes the cardinality. 
For $n\geq 2$, we define $\Omega_n\leq \textup{Homeo}^+(\mathbf{R})$ as the group of homeomorphisms $f\in \textup{Homeo}^+(\mathbf{R})$ satisfying the following: 

\begin{enumerate}

\item $f$ is piecewise linear with breakpoints (where the left and right derivatives do not coincide) in $\mathbf{Z}[\frac{1}{2}]$, and $\mathbf{Z}[\frac{1}{2}]\cdot f=\mathbf{Z}[\frac{1}{2}]$.

\item $f$ commutes with the translation $t\mapsto t+1$.

\item For each $x\in \mathbf{R}\setminus \mathbf{Z}[\frac{1}{2}]$, we have $\log_2(f'(x))\cong (x,x\cdot f)_{\mathbf{Z}} \text{ (mod }n)$.

\end{enumerate}
Here $f'(x)$ denotes the derivative of $f$ at $x$.
We alert the reader to the fact that when going from right actions to derivatives (the latter using function notation), the order of composition needs to be consistent. For instance, the derivative of the element $fg$ at $x$ is $(g\circ f)'(x)$.
One may check that $\Omega_n$ is indeed a subgroup of $\textup{Homeo}^+(\mathbf{R})$. Let us verify that the elements in Definition \ref{ourgroup} are closed under composition (the facts that $\Omega_n$ contains the identity and is closed under inversion are elementary and left to the reader). 
Let $f,g\in \Omega_n$. 
The element $fg$ clearly satisfies conditions $(1),(2)$ of Definition \ref{ourgroup}.
Condition $(3)$ for $fg$ is the assertion that for all $x\in \mathbf{R}\setminus \mathbf{Z}[\frac{1}{2}]$:  $$(x,x\cdot fg)_{\mathbf{Z}}\cong \log_2((g\circ f)'(x))\text{ (mod }n)$$
Since $f,g\in \Omega_n$, we know that for each $x\in \mathbf{R}\setminus \mathbf{Z}[\frac{1}{2}]$: $$(x,x\cdot f)_{\mathbf{Z}}\cong \log_2(f'(x))\text{ (mod }n)\qquad (x,x\cdot g)_{\mathbf{Z}}\cong \log_2(g'(x))\text{ (mod }n)$$
Note that $(x,y)_{\mathbf{Z}}+(y,z)_{\mathbf{Z}}=(x,z)_{\mathbf{Z}}$ for all $x,y,z\in \mathbf{R}\setminus \mathbf{Z}$.
The chain rule implies that: $$\log_2((g\circ f)'(x))=\log_2(f'(x))+\log_2(g'(f(x)))\cong (x,x\cdot f)_{\mathbf{Z}}+(x\cdot f,x\cdot fg)_{\mathbf{Z}}\text{ (mod }n)\cong (x,x\cdot fg)_{\mathbf{Z}}\text{ (mod }n)$$
which means that $fg\in \Omega_n$ as required, confirming that $\Omega_n$ is a group.
Next, the definition implies that: $$\langle t\to t+m\mid m\in \mathbf{Z}\rangle\cap \Omega_n=\langle t\to t+m\mid m\in n\mathbf{Z}\rangle$$
Since elements of $\Omega_n$ commute with integer translations, this action descends to an action of the group:
$$\Gamma_n:=\Omega_n/\langle t\to t+m\mid m\in n\mathbf{Z}\rangle$$ on $\mathbf{S}^1\cong \mathbf{R}/\mathbf{Z}$.
We denote by $\Gamma_n'$ the commutator subgroup of $\Gamma_n$.
\end{definition}

\begin{thm}\label{main}
For each $n\in \mathbf{N}, n\geq 2$, the group $\Gamma_n'$ is uniformly simple, finitely presented, and of type $F_{\infty}$. Moreover, $\mathcal{R}(\Gamma_n')\geq \floor{\frac{n}{2}}$.
\end{thm}

We remark that the most difficult part of this result (beyond the novelty of the idea for the construction itself) is the proof of uniform simplicity of $\Gamma_n'$. The most natural strategies to demonstrate this, including the criteria provided in \cite{GalGis} and \cite{GuelmanLiousse}, fail for conceptual reasons. The main issue is that there are surprising elements in the stabilizer of $0$ in $\Gamma_n'$ that do not emerge as products of commutators of elements in this stabilizer (the details emerge later in the paper). To prove uniform simplicity our strategy requires an intricate analysis of certain homomorphisms produced using two different derivative cocycles, one defined on the rigid stabilizer of an interval in the associated circle action, and the other defined on the entire circle. This is the technical heart of the paper, and our method may also find applications in other settings to prove uniform simplicity. Indeed, the authors believe that our method of analysis of such homomorphisms is novel and will find use even beyond the context of uniform simplicity.

A group $G$ is \emph{$n$-uniformly perfect} if every element can be expressed as a product of at most $n$ commutators of elements in the group,
and it is \emph{uniformly perfect} if it is $n$-uniformly perfect for some $n\in \mathbf{Z}$.
The smallest such number is denoted as the \emph{commutator width of $G$}, or $\mathcal{P}(G)$.
For perfect groups that are not uniformly perfect, we declare $\mathcal{P}(G)=\infty$.
The notion of uniform perfection is of fundamental importance in group theory and topology.
It was conjectured by Ore in $1951$ that finite simple groups are $1$-uniformly perfect (i.e. every element is a commutator), and this was only proven in $2010$
\cite{Ore}. Moreover, uniform perfection has consequences for stable commutator length and bounded cohomology. Indeed, stable commutator length vanishes for all elements of uniformly perfect groups, and hence the group does not admit non-trivial homogeneous quasimorphisms. We refer the reader to Calegari's book \cite{calegari} for details. In this situation, commutator width of the group emerges as a natural invariant. For the class of finitely generated (but not finitely presentable) simple groups, it was demonstrated by Muranov in \cite{Muranov} in 2007 that this width can be arbitrarily large.
However, to our knowledge, it remained open whether this can be the case for finitely presented infinite simple groups.
Our construction also settles this problem:

\begin{thm}\label{mainperfect}
For each $n\in \mathbf{N}, n\geq 2$, we have $\floor{\frac{n}{4}}\leq \mathcal{P}(\Gamma_n')< \infty$.
In particular, $\mathcal{P}(G)$ is unbounded for the class of finitely presented infinite simple groups.
\end{thm}

We remark that the lower bound in Theorem \ref{main} is an immediate corollary of Theorem \ref{mainperfect}.

\subsection{Some fundamental remarks on our construction.}

We encourage the reader to verify using the definition that the stabilizer of $0$ in $\Omega_n$ is naturally isomorphic to the Higman-Thompson group $F_{2^n}$.
(The definition of Higman-Thompson groups is recalled in Subsection \ref{HigmanThompson}). 
This fact (stated in Lemma \ref{HigmanSubgroups}) is an important starting point of our construction, and shall be used repeatedly in this article.
Our construction above is related to our recent construction of finitely presented simple left orderable groups in \cite{HydeLodhaFPSimple}. However, there are significant differences between the groups. In particular, the groups constructed in \cite{HydeLodhaFPSimple} are not uniformly simple. Indeed, they admit a nontrivial homogeneous quasimorphism and as a consequence have infinite commutator width. In particular, they are far from being uniformly perfect or uniformly simple. 
Also, the groups $\{\Gamma_n'\}_{n\geq 2}$ constructed in this paper are not left orderable, since they contain torsion elements, and the groups constructed in \cite{HydeLodhaFPSimple} are left orderable.
Moreover, one faces distinct technical challenges in dealing with the construction of this paper compared to the groups in \cite{HydeLodhaFPSimple}, although the proofs establishing the finiteness property type $F_{\infty}$ are similar.

Since a first order sentence witnesses $\mathcal{R}(G)=n$, our groups emerge as a natural infinite family of pairwise non-elementarily equivalent finitely presented infinite simple groups.
It follows from the definition that the groups $\{\Gamma_n'\}_{n\geq 2}$ naturally embed in Thompson's group $T$, and hence admit actions by $C^{\infty}$-diffeomorphisms 
of the circle (applying the main result of \cite{GhysSergiescu}). In particular, they do not have Kazhdan's property $(T)$ (from a direct application of the results in \cite{NavasKazhdan}, \cite{Cornulier}, \cite{LMBT}). In fact, since the groups embed into Thompson's group $T$, they have the Haagerup property (see \cite{Hughes}). 

\textbf{Acknowledgements:} The authors thank Yves Cornulier, Francesco Fournier-Facio and Matt Brin for their feedback. The second author was partially supported by the NSF CAREER grant 2240136.


\section{Preliminaries.}
All actions in this paper will be right actions. We fix the following notation. For a group $G$ and elements $f,g\in G$, we denote: 
$G':=[G,G], G'':=[G',G'], f^g:=g^{-1}fg, [f,g]:=f^{-1}g^{-1}fg$.
Given a subset $X\subseteq G$, we denote by $\langle \langle X\rangle \rangle_G$ as the smallest normal subgroup in $G$ containing $X$. 
Our convention will be to identify a group with an action on a $1$-manifold without declaring specific notation for the action. 
The nature of the action will be made clear from the context.

Let $M$ be a connected $1$-manifold, and $G\leq \textup{Homeo}^+(M)$.
Given $X\subset M,x\in M$, we refer to the pointwise (or point stabilizers) as $\textup{Stab}_G(X),\textup{Stab}_G(x)$ respectively.
Given $f\in G$, define the \emph{support of $f$} as $\Supp(f)=\{x\in M\mid x\cdot f\neq x\}$.
Given $I\subseteq M$, denote 
the \emph{rigid stabilizer of $I$} as $\textup{Rstab}_{G}(I)=\{f\in G\mid \Supp(f)\subseteq I\}$.
For $x\in M$, whenever they exist, we denote $f'(x),f'_-(x),f'_+(x)$ as the derivative, left derivative and right derivative, respectively.
The action of $G$ on $M$ is \emph{minimal} if all orbits are dense in $M$. 
The action is \emph{proximal} if for every proper nonempty compact subset $U\subset M$ and nonempty open subset $V\subset M$,
there is an element $f\in G$ such that $U\cdot f\subset V$.
An element $f\in \textup{Homeo}^+(\mathbf{R})$ is said to be \emph{$1$-periodic} if it commutes with all integer translations. A subgroup $G\leq \textup{Homeo}^+(\mathbf{R})$ is said to be $1$-periodic if every element of $G$ is $1$-periodic. 
We say that a $1$-periodic group is \emph{$1$-periodically proximal} if for each pair of nonempty open intervals $(u_1,v_1),(u_2,v_2)$ such that $v_1-u_1<1$, there is a group element $f$ such that $((u_1,v_1)+\mathbf{Z})\cdot f\subset (u_2,v_2)+\mathbf{Z}$.
The following was observed in \cite{HydeLodhaFPSimple} (Lemma $2.1$).

\begin{prop}\label{generalminimal}
Let $G\leq \textup{Homeo}^+(\mathbf{R})$ be a $1$-periodic subgroup such that: 
\begin{enumerate}[itemsep=0pt,parsep=0pt]
\item For each $n\in \mathbf{Z}$, every $x\in (n,n+1)$ and every nonempty open set $U\subset (n,n+1)$, there is an $f\in G$ such that $x\cdot f\in U$.
\item $\mathbf{Z}$ is not $G$-invariant.
\end{enumerate}
Then the action of $G$ on $\mathbf{R}$ is minimal. If such an action also does not preserve a Radon measure on $\mathbf{R}$, then it is $1$-periodically proximal.
\end{prop}

\subsection{Finiteness properties of groups.}

The finiteness properties type $\mathbf{F}_{n}$ and type $\mathbf{F}_{\infty}$ are of fundamental importance in geometric group theory since they are quasi-isometry invariants of groups \cite{Alonso},\cite{Geoghegan}. 
A group is said to be \emph{of type $\mathbf{F}_n$} if it admits an 
Eilenberg-Maclane complex with a finite $n$-skeleton, and is \emph{of type $\mathbf{F}_{\infty}$} if it is of type $\mathbf{F}_n$ for each $n\in \mathbf{N}$.
We recall a special case of what is called Brown's criterion
(Proposition $1.1$ in \cite{brown}).

\begin{prop}\label{brownscriterion}
 Let $\Gamma$ be a group that acts on a cell complex $X$
by cell permuting homeomorphisms such that $X$ is contractible, $X/ \Gamma$ has finitely many cells in each dimension, and the pointwise stabilizer of each cell is of type $\mathbf{F}_{\infty}$.
 Then $\Gamma$ is of type $\mathbf{F}_{\infty}$.
\end{prop}

The following is a standard fact (see \cite{Geoghegan} for a proof).

\begin{prop}\label{extensionfiniteness}
 Consider a group extension $1\to N\to G\to H\to 1$. If $N,H$ are of type $\mathbf{F}_{\infty}$ then $G$ also has type $\textbf{F}_{\infty}$. In particular, a finite direct product of type $\textbf{F}_{\infty}$ groups is also of type $\textbf{F}_{\infty}$.
\end{prop}

Given a group $H$ and an isomorphism $\phi:H\to K$, where $K<H$ is a proper subgroup, the group $\langle H, t\mid t^{-1}ht=\phi(h)\text{ for }h\in H\rangle$ is called an \emph{ascending HNN extension} with \emph{base group} $H$. 
The following is a criterion to verify whether a group admits such a structure (See Lemma $3.1$ in \cite{GMSW} for a proof):
\begin{lem}\label{AHNN}
Let $G$ be a group that satisfies the following. There exist subgroups $H_1< H_2< G$ and an element $f\in G$ such that $f^{-1}H_2f=H_1$,
and $\langle H_2,f\rangle=G$.
Then $G$ admits the structure of an ascending HNN extension with base group $H_2$.
\end{lem}

Finiteness properties of ascending HNN extensions are well behaved (see the end of Section $2$ in \cite{BDH}):
\begin{prop}\label{AHNNF}
Let $G$ be an ascending HNN extension with base group $H$. If $H$ has type $\mathbf{F}_{\infty}$, then $G$ has type $\mathbf{F}_{\infty}$.
\end{prop}

Finally, we state the following standard fact concerning finiteness properties of quotients (see Theorem $7.2.21$ in \cite{Geoghegan}).

\begin{thm}\label{Finftyquotients}
Let $G$ be a group of type $\mathbf{F}_{\infty}$ and let $N\leq G$ be a normal subgroup which is also of type $\mathbf{F}_{\infty}$. Then the quotient $G/N$ is also of type $\mathbf{F}_{\infty}$.
\end{thm}

\subsection{Weak generating sets.}

Given a group $G$ and a subset $S\subseteq G$, recall that that $S$ \emph{weakly generates $G$} or is a \emph{weak generating set of $G$} if the image of $S$ under the abelianization map generates the abelianization.
A weak generating set $S$ is said to be \emph{minimal} if no proper subset of it is a weak generating set.
The following is a straightforward observation. 

\begin{lem}\label{roots}
Let $G$ be a group so that $G/[G,G]\cong \mathbf{Z}^k$ for some $k\in \mathbf{N}$ and let $S\subset G, |S|=k$ be such that there is a homomorphism of $G$ with image an abelian group, so that the image of $S$ under this homomorphism generates a copy of $\mathbf{Z}^{k}$.
Then the group $H=\langle S,[G,G]\rangle$ is a finite index subgroup of 
$G$ and $S$ is a weak generating set for $H$.
\end{lem}

\begin{proof}
By our hypothesis, there is a homomorphism $G\to K$ where $K$ is an abelian group such that the image of $S$ under this freely generates a copy of $\mathbf{Z}^{|S|}$. Since this must pass through the abelianization map, this means that the image of $S$ in the abelianization map also generates a copy of $\mathbf{Z}^{|S|}=\mathbf{Z}^k$, and indeed $K=G/[G,G]$. Since $|S|$ equals the rank of the torsion-free abelianization $G/[G,G]$, the image of $S$ under this map is a finite index subgroup of the abelianization, and whose inverse image in $G$ is the finite index subgroup $H=\langle S,[G,G]\rangle$. The fact that $S$ is a weak generating set for $H$ follows immediately.
\end{proof}

\subsection{Conjugacy classes and bounded coverings.}\label{BCsubsec}

Given a group $G$ and an element $g\in G$, recall that $C_g(G)$ denotes the conjugacy class of $g$ in $G$.
Given a subgroup $H\leq G$, we denote by $C_H(G):=\bigcup_{h\in H} C_h(G)$.
The group $G$ is said to be \emph{$m$-boundedly covered by $H$} if 
for each $g\in G$ there exist $g_1,...,g_k\in C_H(G), k\leq m$ such that 
$g=g_1...g_k$.
The group $G$ is said to be \emph{boundedly covered by $H$} if it is $m$-boundedly covered by $H$ for some $m\in \mathbf{N}$.

In a group $G$ a subgroup $K$ is said to be \emph{$m$-boundedly covered by a subgroup $H$} if 
for each $g\in K$ there exist $g_1,...,g_k\in C_H(G), k\leq m$ such that 
$g=g_1...g_k$.
In a group $G$ a subgroup $K$ is said to be \emph{boundedly covered by a subgroup $H$} if there exists an $m\in \mathbf{N}$ such that $K$ is $m$-boundedly covered by $H$. Note that here we consider conjugacy in the group $G$.
The following is a straightforward lemma that follows from the definitions, and that will be among the ingredients in the proof of uniform simplicity of our groups. In what follows, $\leq_{\text{ f.i. }}$ denotes a subgroup of finite index.

\begin{lem}\label{BC}
Let $G$ be a simple group containing subgroups: $$H_1\leq H_3\leq_{f.i.} H_2\leq G$$
(where $\leq_{f.i.}$ denotes a finite index subgroup) such that:
\begin{enumerate}
\item $H_1$ is $l$-uniformly simple for some $l\in \mathbf{N}$.
\item For each $f\in G\setminus \{id\}$, there exist $\alpha_1, \alpha_2\in G$ such that $[[f,\alpha_1],\alpha_2]\in H_1\setminus \{id\}$.
\item $H_3$ is $m_1$-boundedly covered by $H_1$.
\item There exists $m_2\in \mathbf{N}$ such that for each $g\in G\setminus \{id\}$, there exist 
$g_1,...,g_k\in C_{H_1}(G), k\leq m_2$ and $h\in H_2$ such that 
$g=hg_1...g_k$.
\end{enumerate}
Then $G$ is uniformly simple.
\end{lem}

\begin{proof}
Since we consider conjugacy in the ambient group $G$, for simplicity of notation we shall denote occurrences of $C_f(G)$ as simply $C_f$ in this proof.
We check the definition of uniformly simplicity.
Fix a pair of nontrivial $f,g\in G$.
Let $g=hg_1...g_k$ as in part $(4)$ of the hypothesis.
Since $H_3$ is finite index in $H_2$, we choose representatives $\{p_1,...,p_j\}$ for the left cosets of $H_3$ in $H_2$, and for some $j\in \mathbf{N}$.  
It follows that $h=p_ih_1$ for some $1\leq i\leq j$ and $h_1\in H_3$.
So we obtain: $$g=hg_1...g_k=p_ih_1g_1...g_k$$

First observe that $\alpha=[[f,\alpha_1],\alpha_2]$ from part $(2)$ can be expressed as a product of four elements in $C_f\cup C_{f^{-1}}$.
Next, from part $(1)$, each $g_i\in H_1$ can be expressed as a product of at most $l$ elements in $C_{\alpha}\cup C_{\alpha^{-1}}$.
This means that the element $g_1...g_k$ can be expressed as a product of at most $4lk\leq 4lm_2$ elements in 
$C_f\cup C_{f^{-1}}$. 

A similar argument, using the fact that $H_3$ is $m_1$-boundedly covered by $H_1$,
shows that $h_1$ can be expressed as a product of at most $4lm_1$ elements in 
$C_f\cup C_{f^{-1}}$. Combining these two facts and part $(4)$ of the hypothesis, we conclude that the element $h_1g_1...g_k$
can be expressed as a product of at most $4lm_1+4lm_2$ elements in 
$C_f\cup C_{f^{-1}}$.

Consider $\alpha=[[f,\alpha_1],\alpha_2]\in H_1\setminus \{id\}$ as defined above. 
Fix $\beta\in H_1\setminus \{id\}$ (independent of our choices of $f,g$).
Recall that we chose representatives $\{p_1,...,p_j\}$ for the left cosets of $H_3$ in $H_2$.
Since $G$ is simple, for each $p_i$ there is a $q_i\in \mathbf{N}$ such that $p_i$
is a product of at most $q_i$ many elements in $C_{\beta}\cup C_{\beta^{-1}}$.
Fix $q=\text{max}\{q_1,...,q_j\}$.
Moreover, as previously observed $\alpha$ is a product of $4$ elements in $C_f\cup C_{f^{-1}}$
and so from the first part of the hypothesis $\beta$ is a product of at most $4l$ elements in $C_f\cup C_{f^{-1}}$.
It follows that each $p_i$ is a product of at most $4lq$ many elements in $C_f\cup C_{f^{-1}}$.
Then we obtain that the element: $$g=hg_1...g_k=p_ih_1g_1...g_k$$
is a product of at most $4lq+4lm_1+4lm_2$ elements in $C_f\cup C_{f^{-1}}$.
Since this bound is independent of the choice of $f,g\in G\setminus \{id\}$, we are done.
\end{proof}

Another straightforward but useful lemma is the following. 

\begin{lem}\label{BC1}
Let $G$ be a group and $H\leq G$ a subgroup such that $H$ admits a weak generating set $\{f_1,...,f_k\}$, and so that for each $1\leq i\leq k$ we have $f_i\in C_{H'}(G)$.
Then $H$ is $(k+1)$-boundedly covered by $H'$ in $G$.
\end{lem}

\begin{proof}
By the definition of a weak generating set, each element of $H$ can be written as an expression of the form 
$f_1^{l_1}...f_{k}^{l_k}h$ where $l_1,...,l_k\in \mathbf{Z}$ and $h\in H'$.
Since each of the elements $f_1^{l_1},...,f_{k}^{l_k},h$ is conjugate to an element of $H'$ by an element of $G$,
our conclusion follows.
\end{proof}

\subsection{The Higman-Thompson groups \texorpdfstring{$F_n$}.}\label{HigmanThompson}
For each $n\in \mathbf{N},n\geq 2$, the Higman-Thompson group $F_n$ is the group of orientation-preserving piecewise linear homeomorphisms $f:[0,1]\to [0,1]$ so that the slopes of $f$, whenever they exist, lie in $\{n^m\mid m\in \mathbf{Z}\}$ and breakpoints lie in $\mathbf{Z}[\frac{1}{n}]\cap (0,1)$. 
We shall refer to this as the \emph{standard action} of $F_n$.
In \cite{brown}, Brown proved the following. 

\begin{thm}\label{brownthm}
For each $n\geq 2$, the group $F_n$ is of type $\mathbf{F}_{\infty}$.
\end{thm}

In \cite{brown}, Brown isolated the presentation: 
$F_n:=\langle(f_i)_{i\in \mathbf{N}}\mid f_i^{-1} f_jf_i=f_{j+n-1}\text{ for }i<j\rangle$
and proved:

\begin{thm}\label{abelianization}
The group $F_n$ admits a generating set of cardinality $n$, and its abelianization is $\mathbf{Z}^n$. The derived subgroup $F_n'$ is simple and every proper normal subgroup of $F_n$ contains $F_n'$. Moreover, if $F_n'\leq H\leq F_n$ then the derived subgroup of $H$ also equals $F_{n}'$.
\end{thm}

The following standard facts also emerge in \cite{brown}, and were also proved in \cite{HydeLodhaFPSimple} (see Lemma $3.3$). 
Recall that if $M$ is a connected $1$-manifold and $G\leq \textup{Homeo}^+(M)$, the action is \emph{proximal} if for every proper nonempty compact subset $U\subset M$ and nonempty open subset $V\subset M$,
there is an element $f\in G$ such that $U\cdot f\subset V$.

\begin{lem}\label{Fnproximal}
The standard actions of $F_n,F_n'$ on $(0,1)$ are proximal. Moreover, for each $a,b\in [0,1]\cap \mathbf{Z}[\frac{1}{n}]$ and $a<b$, it holds that
$\textup{Rstab}_{F_n}([a,b])\cong F_n$.
\end{lem}

$F_{n}^c\leq F_{n}$ is the subgroup consisting of elements of $F_{n}$ whose germs at $0$ and $1$ are trivial. Equivalently, these are elements $f\in F_{n}$ for which there is some compact interval $I\subset (0,1)$ such that $Supp(f)\subset I$.
The following is a well known fact about this group \cite{brown}.
\begin{lem}\label{Fncompactab}
The derived subgroup of $F_{n}^c$ coincides with $F_{n}'$ and the abelianization of $F_{n}^c$ is isomorphic to $\mathbf{Z}^{n-2}$.
\end{lem}

In \cite{GalGis} (Theorem $1.1$), it was shown that groups that admit a proximal order-preserving action on a linearly ordered set, for which all elements have bounded support, satisfy that their derived subgroups are $6$-uniformly simple. Since $F_n'$ admits such an action on $(0,1)$ and $F_n'=F_n''$, we have the following:
\begin{thm}\label{commwid}
(Gal, Gismatullin \cite{GalGis}) The group $F_n'$ is $6$-uniformly simple.
\end{thm}


Define the map: $$\theta_n:\mathbf{Z}\left[\frac{1}{n}\right]\to \mathbf{Z}/(n-1)\mathbf{Z}\qquad \frac{k}{n^m}\mapsto k (\text{mod }n-1)$$ 
We observed the following in \cite{HydeLodhaFPSimple} (Proposition $3.5$).

\begin{prop}\label{Fntrans}
Fix $n\in \mathbf{N},n\geq 2$. For each $k\geq 1$, consider two linearly ordered $k$-tuples $(x_1,...,x_k)$ and $(y_1,...,y_k)$ in $\mathbf{Z}[\frac{1}{n}]\cap (0,1)$ (here the linear order is induced from the natural one on $\mathbf{R}$). Then the following are equivalent:
\begin{enumerate}
\item There is an element $f\in F_n$ such that $x_i\cdot f= y_i$ for each $1\leq i\leq k$.
\item There is an element $f\in F_n'$ such that $x_i\cdot f= y_i$ for each $1\leq i\leq k$.
\item $\theta_n(x_i)=\theta_n(y_i)$ for each $1\leq i\leq k$. 
\end{enumerate}
In particular, the orbits of the action of $F_n,F_n',F_n^c$ are precisely the fibers of $\theta_n$ in $(0,1)$, and the set $\{\frac{1}{n},...,\frac{n-1}{n}\}$ is a transversal.
\end{prop}

The following lemma was also proved in \cite{HydeLodhaFPSimple} (Lemma $3.7$).

\begin{lem}\label{lemorbits2}
If $f:\mathbf{R}_{\geq 0}\to \mathbf{R}_{\geq 0}$ is a piecewise linear homeomorphism with slopes in 
$\{n^k\mid k\in \mathbf{Z}\}$ and breakpoints in $\mathbf{Z}[\frac{1}{n}]$, then $\theta_n(x\cdot f)=\theta_n(x)$ for each $x\in \mathbf{Z}[\frac{1}{n}], x\geq 0$.
\end{lem}

Finally, we shall need one more ingredient which was also proved in \cite{HydeLodhaFPSimple} (Lemma $3.16$).
For a closed interval $J\subset (0,1)$, we denote by $\textup{Rstab}_{F_n}^c(J)$ the group consisting of elements 
in $\textup{Rstab}_{F_n}(J)$ whose germs at $inf(J),sup(J)$ are trivial.

\begin{lem}\label{FPlem2}
Consider the standard action of $F_{n}$ on $[0,1]$.
Let $I, J\subset [0,1)$ be closed intervals with endpoints in $\mathbf{Z}[\frac{1}{n}]$ such that $J\subset \Int(I)$. 
Then $\textup{Rstab}_{F_n}(I)'\cap \textup{Rstab}_{F_n}^c(J)=\textup{Rstab}_{F_n}(J)'$.
\end{lem}

We define the \emph{$1$-periodic action of $F_{n}$} as the unique $1$-periodic subgroup of $\textup{Homeo}^+(\mathbf{R})$ which pointwise fixes $\mathbf{Z}$ and whose restriction to $[0,1]$ coincides with the standard action of $F_n$ defined above. Note that since this is $1$-periodic, it will act in a similar way on each interval $[n,n+1]$ for $n\in \mathbf{Z}$.

\section{The groups \texorpdfstring{$\{\Omega_n\}_{n\geq 2}$}.}
In this section we shall prove various structural results concerning the groups $\{\Omega_n\}_{n\geq 2}$, which shall help us prove our main theorem.
The following lemma follows immediately from our definitions.
We alert the reader to the fact that $\mathbf{Z}[\frac{1}{2^n}]=\mathbf{Z}[\frac{1}{2}]$.

\begin{lem}\label{HigmanSubgroups}
The stabilizer of $0$ in $\Omega_n$ coincides with the $1$-periodic action of the Higman-Thompson group $F_{2^n}$.
\end{lem}
We will denote this $1$-periodic subgroup as $F_{2^n}$ throughout the rest of the article.
Also, throughtout this article, we fix the notation:
$$t_n:\mathbf{R}\to \mathbf{R}\qquad x\cdot t_n=x+n\qquad \Lambda_n=\langle t_n\rangle=\{ x\to x+m\mid m\in n\mathbf{Z}\}$$
\begin{lem}\label{center}
The center of $\Omega_n$, denoted as $\mathcal{Z}(\Omega_n)$, coincides with $\Lambda_n$.
\end{lem}

\begin{proof}
It is easy to see that the centralizer of the $1$-periodic action of $F_{2^n}$ in $\textup{Homeo}^+(\mathbf{R})$ coincides with the cyclic group of integer translations. Indeed, this also centralizes $\Omega_n$ in $\textup{Homeo}^+(\mathbf{R})$. We conclude by the fact (mentioned after Definition \ref{ourgroup}) that: $$\langle t\to t+m\mid m\in \mathbf{Z}\rangle\cap \Omega_n=\langle t\to t+m\mid m\in n\mathbf{Z}\rangle=\Lambda_n$$
\end{proof}

The following is a key lemma that allows us to study the orbits of the action of $\Omega_n$ on $\mathbf{R}$.
\begin{lem}\label{GammaOrbits}
Let $f\in \Omega_n$ and $x\in [0,1]\cap \mathbf{Z}[\frac{1}{2}]$ satisfy $y=x\cdot f\in [0,1]$. Then we have $\theta_{2^n}(x)=\theta_{2^n}(y)$.
\end{lem}

\begin{proof}
Switching $f$ to $f^{-1}$ and $x$ to $y$ if needed, assume that $z_1=0\cdot f\geq 0$.
Note that $1\cdot f=1+z_1$.
Let $z_2\in [0,1]$ be such that $z_2\cdot f=1$.
Define $g:\mathbf{R}_{\geq 0}\to \mathbf{R}_{\geq 0}$ as:
$$t\cdot g:=\begin{cases}t\cdot f-z_1\text{ if }t\in [0,z_2]\\
\frac{(t\cdot f)-1}{2}+1-z_1\text{ if }z_2\leq t\leq 1\\
t-\frac{z_1}{2}\text{ if }t\geq 1
\end{cases}$$
We claim that $g$ is a piecewise linear map $\mathbf{R}_{\geq 0}\to \mathbf{R}_{\geq 0}$ with slopes in $\{(2^n)^k\mid k\in \mathbf{Z}\}$ and breakpoints in $\mathbf{Z}[\frac{1}{2^n}]=\mathbf{Z}[\frac{1}{2}]$. 
The condition on breakpoints is clear, and the condition on slopes is apparent for points in $((0,z_2)\cup (1,\infty))\setminus \mathbf{Z}[\frac{1}{2}]$.
For $p\in (z_2,1)\setminus \mathbf{Z}[\frac{1}{2}]$ note that $\log_2(f'(p))\cong 1\mod n$ by definition of $f$ being in $\Omega_n$, and so the map $t\cdot g=\frac{(t\cdot f)-1}{2}+1-z_1$ satisfies that for each $q\in  (z_2,1)\setminus \mathbf{Z}[\frac{1}{2}]$, we have that $\log_{2}(g'(q))=0\mod n$ and so $g'(p)$ is of the form $(2^n)^m$ for some $m\in \mathbf{Z}$.

Applying Lemma 
\ref{lemorbits2}, we get $\theta_{2^n}(p)=\theta_{2^n}(p\cdot g)$ for each $p\in \mathbf{Z}[\frac{1}{2}], p\geq 0$.
So we get: $$\theta_{2^n}(1)=\theta_{2^n}(1\cdot g)= \theta_{2^n}\left(1-\frac{z_1}{2}\right)=\theta_{2^n}(1)+\theta_{2^n}\left(\frac{z_1}{2}\right)$$
which implies that $\theta_{2^n}(\frac{z_1}{2})=0$ and so $\theta_{2^n}(z_1)=2\cdot \theta_{2^n}(\frac{z_1}{2})=0$.
For $x\in [0,z_2]$, we obtain: $$\theta_{2^n}(x)=\theta_{2^n}(x\cdot g)=\theta_{2^n}(x\cdot f-z_1)=\theta_{2^n}(x\cdot f)-\theta_{2^n}(z_1)=\theta_{2^n}(x\cdot f)=\theta_{2^n}(y)$$
\end{proof}

\subsection{Special elements in \texorpdfstring{$\{\Omega_n\}_{n\geq 2}$}.}\label{specialelementsconstruction}

An element $g\in \Omega_n$ is said to be \emph{special} if $0\cdot g\in (0,1)$.
Here is an explicit example of such an element.
For $I:=\left[-\frac{1}{2^n}, \frac{1}{2^{n-1}}\right]$, we define $\lambda:I\to I$ as follows:

$$\lambda:=\begin{cases}
[-\frac{1}{2^n},\frac{1-2^n}{2^{2n}}]\mapsto [-\frac{1}{2^n},0] & \text{ linear with slope } 2^n.\\
[\frac{1-2^n}{2^{2n}},0]\mapsto [0, \frac{2^{n}-1}{2^{2n-1}}]&\text{ linear with slope }2.\\
[0,\frac{1}{2^{n-1}}]\mapsto [\frac{2^{n}-1}{2^{2n-1}},\frac{1}{2^{n-1}}]
&\text{ linear with slope }\frac{1}{2^n}.
\end{cases}$$




Define the $1$-periodic homeomorphism $\tau:\mathbf{R}\to \mathbf{R}$ satisfying $\Supp(\tau)=\Int(I)+\mathbf{Z}$ and $\tau\restriction I=\lambda$. 
Here $\tau\restriction I=\lambda$ denotes that for all $x\in I$, we have $x\cdot \tau=x\cdot \lambda$.
Clearly, $\tau\in \Omega_n$ is a special element. 


A nice consequence of the Lemma \ref{GammaOrbits} is that we obtain a finite generating set for $\Omega_n$.

\begin{lem}\label{Omegagenset}
Let $f\in \Omega_n$ be a special element. Then $\Omega_n=\langle f, F_{2^n}\rangle$.
\end{lem}

\begin{proof}
Observe that $F_{2^n}\restriction (0,1)$ is minimal (from Lemma \ref{Fnproximal}). This fact, together with the existence of special elements as demonstrated above, means that Proposition \ref{generalminimal} is applicable to the group $\langle f,F_{2^n}\rangle$. It follows that $\langle f,F_{2^n}\rangle$ acts minimally on $\mathbf{R}$.
Let $g\in \Omega_n$ be a nontrivial element. It follows from minimality that there is an 
$f_1\in \langle f,F_{2^n}\rangle$ such that $0\cdot gf_1\in (0,1)$.
From Lemma \ref{GammaOrbits}, $\theta_{2^n}(0\cdot gf_1)=\theta_{2^n}(0\cdot f)$.
It follows from Proposition \ref{Fntrans} that there is an $f_2\in F_{2^n}$ such that 
$0\cdot gf_1f_2=0\cdot f$. Therefore, 
$0\cdot gf_1f_2f^{-1}=0$ and so $gf_1f_2f^{-1}\in F_{2^n}$ from Lemma \ref{HigmanSubgroups}.
This means that $g\in\langle f,F_{2^n}\rangle$, finishing the proof.
\end{proof}

Recall from the previous section that a $1$-periodic group is \emph{$1$-periodically proximal} if for each pair of nonempty open intervals $(u_1,v_1),(u_2,v_2)$ such that $v_1-u_1<1$, there is a group element $f$ such that $((u_1,v_1)+\mathbf{Z})\cdot f\subset (u_2,v_2)+\mathbf{Z}$.

\begin{lem}\label{periodproximal0}
The action of $\Omega_n$ on $\mathbf{R}$ is $1$-periodically proximal.
\end{lem}

\begin{proof}
It suffices to check the full hypothesis of Proposition \ref{generalminimal}. 
Recall from Lemma \ref{Omegagenset} that $\Omega_n=\langle \tau,F_{2^n}\rangle$
where $\tau$ is the special element constructed above.
We know that $F_{2^n}\restriction (0,1)$ is minimal (from Lemma \ref{Fnproximal}) and hence the first two conditions in Proposition \ref{generalminimal} are applicable to $\Omega_n=\langle \tau,F_{2^n}\rangle$.
We just need to check that the action does not preserve a Radon measure on $\mathbf{R}$, which was the final ingredient in the hypothesis of Proposition \ref{generalminimal}.
In the presence of such an invariant measure and the absence of global fixed points, 
there exists a so called translation number homomorphism $\Omega_n\to \mathbf{R}$ which is nontrivial and for which each element of the group that admits a fixed point must lie in the kernel.
(See the discussion on page $128$ of \cite{GOD} for details).
Since the group $\Omega_n$ is generated by $\tau$ and $F_{2^n}$, and since all elements in this generating set admit fix points in $\mathbf{R}$, this is a contradiction.
\end{proof}

\begin{lem}\label{specialelementsgalore}
For each open interval $I$ containing $0$, one can find a special element $f\in \Omega_n'$ 
such that $Supp(f)\subset I+\mathbf{Z}$ and $f$ is conjugate to an element of $F_{2^n}'$.
\end{lem}

\begin{proof}
It suffices to consider the case where $I\subset (-\frac{1}{4},\frac{1}{4})$.
Using Lemma \ref{periodproximal0}, we find an element $g_1\in \Omega_n$ such that $I\cdot g_1\subset (0,1)$.
Next, we find an element $g_2\in F_{2^n}'$ such that $0\cdot g_1g_2>0\cdot g_1$ and $Supp(g_2)\subset I\cdot g_1+\mathbf{Z}$.
Then $f=g_1g_2g_1^{-1}$ is the required element.
\end{proof}

We can finally show a strengthening of Proposition \ref{periodproximal0}.

\begin{prop}\label{periodproximal}
The action of $\Omega_n'$ on $\mathbf{R}$ is $1$-periodically proximal.
\end{prop}

\begin{proof}
We will show this for the group $\langle f,F_{2^n}'\rangle$
where $f\in \Omega_n'$ is any special element constructed in Lemma \ref{specialelementsgalore}.
Since $\langle f,F_{2^n}'\rangle\leq \Omega_n'$, this will suffice.
We must check the hypothesis of Proposition \ref{generalminimal} for the group $\langle f,F_{2^n}'\rangle$.
Observe that $F_{2^n}'\restriction (0,1)$ is minimal (from Lemma \ref{Fnproximal}) and hence the first two conditions in Proposition \ref{generalminimal} are applicable to the group. 
It remains to check that the action does not preserve a Radon measure on $\mathbf{R}$. This is the same argument as in the proof of Lemma \ref{periodproximal0}, namely that each generator in the generating set $f,F_{2^n}'$ admits a fixed point in $\mathbf{R}$ and hence the group does not admit a nontrivial translation number homomorphism.
\end{proof}

This can be used to slightly strengthen Lemma \ref{specialelementsgalore}, using the same proof as before, but choosing the conjugating element
to be in $\Omega_n'$ using the fact that the action of $\Omega_n'$ on $\mathbf{R}$ is $1$-periodically proximal.

\begin{lem}\label{specialelementsgalore1}
For each open interval $I\subset (\frac{-1}{4},\frac{1}{4})$ containing $0$, one can find a special element $f\in \Omega_n'$ 
such that $Supp(f)\subset I+\mathbf{Z}$ and $f$ is conjugate to an element of $F_{2^n}'$ by an element of $\Omega_n'$.
\end{lem}

Another useful lemma is the following.

\begin{lem}\label{specialelementsgalore2}
Let $x\in (0,1)$ and let $y\in (x-1,x)$. Then we can find an element $f\in \Omega_n'$ 
such that $x\cdot f=x$ and $y\cdot f\in (0,1)$.
\end{lem}

\begin{proof}
If $y\in (0,x)$ then we can simply choose the identity element, and so assume that $y\in (x-1,0)$.
Using Lemma \ref{specialelementsgalore1}, we find a special element $g_1\in \Omega_n'$ such that $Supp(g_1)\subset (x-1,x)+\mathbf{Z}$.
Next, using Proposition \ref{Fntrans} we find $g_2\in F_{2^n}'$ such that $x\cdot g_2=x$ and $y\cdot g_2 \in (0\cdot g_1^{-1},0)$. It follows that 
$x\cdot g_2g_1=x$ and $y\cdot g_2g_1\in (0,1)$ and so $f=g_2g_1$ is the required element.
\end{proof}


\subsection{The orbit structure of the action of $\Omega_n'$ on $\mathbf{Z}[\frac{1}{2}]$.}

\begin{prop}\label{omegaorbits}
The action of $\Omega_n$ on $\mathbf{Z}[\frac{1}{2}]$ has precisely $2^n-1$ orbits, and $\{\frac{1}{2^n},..., \frac{2^n-1}{2^n}\}$ is a transversal for these. The same holds for the action of $\Omega_n'$, and hence the orbits of $\Omega_n,\Omega_n'$ on $\mathbf{Z}[\frac{1}{2}]$ coincide.
\end{prop}

\begin{proof}
First, note for each $x\in \mathbf{Z}[\frac{1}{2}]$, from minimality there is an $f\in \Omega_n$ such that $x\cdot f\in (0,1)\cap \mathbf{Z}[\frac{1}{2}]$. From Proposition \ref{Fntrans}, $\{\frac{1}{2^n},..., \frac{2^n-1}{2^n}\}$ is a transversal for the orbits of the action of $F_{2^n}$ on $(0,1)\cap \mathbf{Z}[\frac{1}{2}]$. So we can find $g\in F_{2^n}$ such that $(x\cdot f)\cdot g\in \{\frac{1}{2^n},..., \frac{2^n-1}{2^n}\}$. Therefore, each orbit in the action of $\Omega_n$ on $\mathbf{Z}[\frac{1}{2}]$ meets $\{\frac{1}{2^n},..., \frac{2^n-1}{2^n}\}$. Using the same argument, we also conclude this for $\Omega_n'$.

Now assume by way of contradiction that there are $1\leq i<j\leq 2^n-1$ such that there is an
$f\in \Omega_n$ such that $\frac{i}{2^n}\cdot f=\frac{j}{2^n}$. Note that $0\cdot f\in (\frac{j}{2^n}-1,\frac{j}{2^n})$.
Using Lemma \ref{specialelementsgalore2}, we choose an element $f_1\in \Omega_n'$ such that: $$\frac{j}{2^n}\cdot f_1=\frac{j}{2^n}\qquad 0\cdot ff_1\in (0,\frac{j}{2^n})$$ Applying Lemma \ref{specialelementsgalore}, we find a special element $\lambda\in \Omega_n'$ such that 
$Supp(\lambda)\subset (\frac{j}{2^n}-1,\frac{j}{2^n})+\mathbf{Z}$. Using Proposition \ref{Fntrans} and Lemma \ref{GammaOrbits}, we can find an element $f_2\in F_{2^n}'$ such that: $$(0\cdot ff_1)\cdot f_2=0\cdot \lambda\qquad \frac{j}{2^n}\cdot f_2=\frac{j}{2^n}$$ It follows that $0\cdot ff_1f_2\lambda^{-1}=0$ and so $g=ff_1f_2\lambda^{-1}\in F_{2^n}$. Moreover, since $f_1,f_2,\lambda$ fix $\frac{j}{2^n}$, we obtain that $\frac{i}{2^n}\cdot g=\frac{j}{2^n}$. This contradicts the fact that $\{\frac{1}{2^n},..., \frac{2^n-1}{2^n}\}$ is a transversal for the orbits of the action of $F_{2^n}$ on $(0,1)\cap \mathbf{Z}[\frac{1}{2}]$.
\end{proof}

We record the following lemma which is an immediate consequence of Proposition \ref{omegaorbits}.

\begin{lem}\label{gammaorbits}
Identify $\mathbf{S^1}\cong \mathbf{R}/\mathbf{Z}$ set-theoretically with $[0,1)$ and consider the action of $\Gamma_n',\Gamma_n$ on $[0,1)$.
The orbits of the actions of $\Gamma_n'$ and $\Gamma_n$ on $[0,1)\cap \mathbf{Z}[\frac{1}{2}]$ coincide.
\end{lem}

The following technical lemma will also be used later in the article.

\begin{lem}\label{zerofix}
Let $f\in \Omega_n'$ be such that $0\cdot f=0$.
Then we can find $g_1,g_2\in \Omega_n'$ such that $g_1$ and $g_2$ are conjugate to elements of $F_{2^n}'$ and $fg_1g_2$ fixes pointwise a neighborhood of $0$.
\end{lem}

\begin{proof}
Let $h$ be a special element with support in $(-\frac{1}{4},\frac{1}{4})+\mathbf{Z}$ and let $x=0\cdot h, y=1\cdot h^{-1}$.
Note that since $h$ fixes a point in $(0,1)$, it holds that $0<x<y<1$.
Now let $k_1,k_2\in F_{2^n}'$ be such that:
\begin{enumerate}
\item $Supp(k_1), Supp(k_2)\subseteq (x,y)+\mathbf{Z}$.
\item The right slope of $k_1$ at $x$ equals the right slope of $f$ at $0$.
\item The left slope of $k_2$ at $y$ equals the left slope of $f$ at $1$.
\end{enumerate}
Then the required elements are $g_1=h k_1^{-1}h^{-1}, g_2=h^{-1}k_2^{-1}h$.
\end{proof}

\subsection{Large subgroups of $\{\Omega_n\}_{n\geq 2}$ and the simplicity of $\{\Gamma_n'\}_{n\geq 2}$.}
A subgroup $G\leq \Omega_n$ is said to be \emph{large} if:
\begin{enumerate}
\item $G$ contains as a subgroup the $1$-periodic copy of $F_{2^n}'$.
\item $G$ contains a special element of $\Omega_n$.
\end{enumerate}
We remark that it is a consequence of Proposition \ref{generalminimal} that large subgroups of $\Omega_n$ act minimally on $\mathbf{R}$.
Let $G\leq \Omega_n$ be a large subgroup. 
We say that a normal subgroup $N\unlhd G$ is \emph{non-central} if $N$ contains a nontrivial element which is not some integer translation. 

\begin{lem}\label{cornulier}
Let $G\leq \Omega_n$ be a large subgroup and $N\leq G$ a non-central normal subgroup, then the following holds:
\begin{enumerate}[itemsep=0pt,parsep=0pt]
\item $N$ contains the $1$-periodic copy of $F_{2^n}'$.
\item $N$ contains a special element. 
\item $N$ acts minimally on $\mathbf{R}$. 
\item $G$ is generated by $G_{<1}=\{g\in G\mid 0\leq 0\cdot g<1\}$.
\item For each $f\in G$ such that $0\cdot f\in (0,1)$, there is a $g\in N$ such that $0\cdot fg^{-1}=0$.
\item $G/N$ is abelian.
\end{enumerate}
\end{lem}

\begin{proof}
{\bf Part $(1)$}. Fix a nontrivial $f\in N$ that is not an integer translation. So there is an $x\in \mathbf{R}\setminus \mathbf{Z}$ such that $x\cdot f, x-x\cdot f\notin \mathbf{Z}$. We find a small open interval $U,|U|<1$ containing $x$ such that $$(U+\mathbf{Z})\cdot f\cap (U+\mathbf{Z})=\emptyset\qquad (U\cap \mathbf{Z})=\emptyset\qquad ((U\cdot f)\cap \mathbf{Z})=\emptyset.$$ Let $g\in F_{2^n}'\leq G$ be a nontrivial element with support in $U+\mathbf{Z}$. Then $h=[g,f]$ is nontrivial and pointwise fixes $\mathbf{Z}$, so is in $(N\cap F_{2^n})\setminus \{1\}$. Since no nontrivial element of $F_{2^n}$ centralizes $F_{2^n}'$, there is a $k\in F_{2^n}'$ such that $[h,k]\in F_{2^n}'\leq G$ is nontrivial. Since $N$ is normal in $G$, it follows that $[h,k]\in N\cap F_{2^n}'$.
Since $F_{2^n}'$ is simple and $F_{2^n}'\cap N$ is normal and nontrivial in $F_{2^n}'$, it follows that $F_{2^n}'\leq N$.\\

{\bf Part $(2)$}. Using minimality, we choose $f\in G$ such that $0\cdot f\in (0,1)$. Let $U\subset (0,1)$ be an open interval such that $(U\cdot f^{-1})\subset (0,1)$.
Using the minimality of the action of $F_{2^n}'$ on $(0,1)$, we find $g\in F_{2^n}'\leq N$ such that $(0\cdot f)\cdot g\in U$. Then the required element in $N$ is $fgf^{-1}g^{-1}$,
since $0\cdot fgf^{-1}\in (0,1)$ and $(0,1)\cdot g^{-1}=(0,1)$.\\

{\bf Part $(3)$}. We apply Proposition \ref{generalminimal} (just the portion needed for minimality), using the previous parts to verify the hypothesis. \\

{\bf Part $(4)$}. From our hypothesis and part $(1)$, $F_{2^n}'\leq G_{<1}$. From part $(2)$ there is an $f\in G_{<1}$ such that $0\cdot f \in (0,1)$.
Applying Proposition \ref{generalminimal}, the group $H$ generated by $G_{<1}$ acts minimally on $\mathbf{R}$.
We must show that $G=H$.
So given $h\in G$, using minimality of $H$ we can find a $g\in H$ such that $(0\cdot h)\cdot g\in (0,1)$. Therefore, $hg\in H$ implies that $h\in H$, and hence $G=H$.\\

{\bf Part $(5)$}. From part $(2)$, choose $g_1\in N$ such that $0\cdot g_1\in (0,1)$.
By Lemma \ref{GammaOrbits} we know that $\theta_{2^n}(0)=\theta_{2^n}(0\cdot f)=\theta_{2^n}(0\cdot g_1)$. From Proposition \ref{Fntrans}, there is a $g_2\in F_{2^n}'\leq N$ satisfying that $0\cdot g_1=(0\cdot f)\cdot g_2$. It follows that $0\cdot fg_2g_1^{-1}=0$ and hence $g=g_1g_2^{-1}\in N$ is the required element satisfying that $0\cdot fg^{-1}=0$.\\

{\bf Part $(6)$}. Let $\phi:G\to G/N$.
From part $(1)$, $F_{2^n}'\leq N$, and so from Theorem \ref{abelianization} $\phi(F_{2^n}\cap G)$ is abelian. So the proof reduces to the claim: $\phi(G)=\phi(F_{2^n}\cap G)$.
From part $(4)$, $G_{<1}=\{g\in G\mid 0\leq 0\cdot g<1\}$ generates $G$.
Let $f\in G_{<1}$.
Using part $(5)$, find $g\in N$ such that $0\cdot fg^{-1}=0$.
It follows that $\phi(f)=\phi(fg^{-1})$ and $fg^{-1}\in F_{2^n}$. So our claim follows.
\end{proof}




Note that by definition, both $\Omega_n$ and $\Omega_n'$ are large subgroups of $\Omega_n$ and hence Lemma \ref{cornulier} is applicable to both. This can be used to give a quick proof of the simplicity of $\Gamma_n'$.

\begin{prop}\label{SimpleProp}
$\Gamma_n'$ is simple.  
\end{prop}

\begin{proof}
Let $f\in \Gamma_n\setminus \{id\}$, and choose a lift $g\in \Omega_n$ of $f$. Then $\Gamma_n/\langle \langle f \rangle \rangle_{\Gamma_n}=\Omega_n/\langle \langle g,\mathcal{Z}(\Omega_n)\rangle\rangle_{\Omega_n}$ and since the latter is abelian from Lemma \ref{cornulier} part $(6)$, so is the former. We conclude that every proper quotient of $\Gamma_n$ is abelian, hence $\Gamma_n'=\Gamma_n''$. A similar argument shows that every proper quotient of $\Gamma_n'$ is abelian, which combined with the fact that $\Gamma_n'=\Gamma_n''$ shows that $\Gamma_n'$ is simple.
\end{proof}

We will also need the following lemma.
Note that in the latter part of the statement, $F_{2^n}\leq \Gamma_n$ is the natural copy of $F_{2^n}$ acting on $[0,1)\cong \mathbf{S}^1\cong \mathbf{R}/\mathbf{Z}$.

\begin{lem}\label{GammaGen}
Let $f\in \Omega_n'$ be any special element. Then $\Omega_n'=\langle f, (F_{2^n}\cap \Omega_n')\rangle$.
Moreover, if $f'$ is the image of $f$ in $\Gamma_n'$ then $\Gamma_n'=\langle f', (F_{2^n}\cap \Gamma_n')\rangle$.
\end{lem}

\begin{proof}
Let $h\in \Omega_n'$. Note that $G=\langle f, (F_{2^n}\cap \Omega_n')\rangle$ is a large subgroup since it contains $F_{2^n}'$ and a special element.
Using Lemma \ref{cornulier}, we see that $G$ acts minimally on $\mathbf{R}$. In particular, there is an element $g_1\in G$ such that $(0\cdot h)\cdot g_1\in (0,1)$.
Applying Lemma \ref{GammaOrbits} and Proposition \ref{Fntrans}, we find an element $g_2\in F_{2^n}'$ such that $0\cdot hg_1g_2=0\cdot f$.
It follows that $0\cdot hg_1g_2f^{-1}=0$ and hence $hg_1g_2f^{-1}\in (F_{2^n}\cap \Omega_n')\leq G$. Since $f,g_1,g_2\in G$, it follows that $h\in G$ as required.
The second part follows immediately since $\Gamma_n'$ is the image of $\Omega_n'$ under the quotient $\Omega_n\to \Gamma_n$.
\end{proof}

\section{The groups \texorpdfstring{$\{\Theta_n\}_{n\geq 2}$}.}\label{minimalweak}
Recall that $F_{2^n}^c$ is the subgroup of $F_{2^n}$ consisting of elements that pointwise fix a neighborhood of $0$.
Recall from Lemma \ref{Fncompactab} that the abelianization of $F_{2^n}^c$ is $\mathbf{Z}^{2^n-2}$.
In this section we construct a finite index subgroup of $F_{2^n}^c$ and provide a special choice of a weak generating set for it.
We remark that initially this group may appear to be somewhat ad hoc, but it shall emerge as a very natural object in our eventual proof of uniform simplicity.
Recall from Proposition \ref{Fntrans} that the $F_{2^n}^c$-orbits on $ (0,1)\cap \mathbf{Z}[\frac{1}{2}]$ admit $\{\frac{i}{2^n}\mid 1\leq i\leq 2^n-1\}$ as a transversal.
For each $1\leq i\leq 2^n-1$, let $O_i\subset (0,1)\cap\mathbf{Z}[\frac{1}{2}]$ be the $F_{2^n}^c$-orbit of $\frac{i}{2^n}$.
Also recall from Proposition \ref{Fntrans} that $O_i$ is the preimage of $i$ under the map: $$\theta_{2^n}:(0,1)\cap \mathbf{Z}\left[\frac{1}{2}\right]\to \mathbf{Z}/(2^n-1)\mathbf{Z}\qquad \frac{k}{(2^n)^m}\mapsto k (\text{mod }2^n-1).$$
We shall use this orbit structure to build a map $\Xi:F_{2^n}^c\to \mathbf{Z}^{2^n-2}$.
For each $x\in  (0,1)\cap \mathbf{Z}[\frac{1}{2}]$, define: $$D_f(x)=\log_{2^n}f_+^{'}(x) -\log_{2^n}f_-^{'}(x)$$
It is a consequence our definitions that for all but finitely many $x$, $D_f(x)=0$.
For each $1\leq i\leq 2^n-1$, define: 
$$\Xi_i:F_{2^n}^c\to \mathbf{Z}_i\qquad \Xi_i(f)=\sum_{x\in O_{i}} D_f(x)$$
where the $\mathbf{Z}_i$ are isomorphic copies of $\mathbf{Z}$ distinguished by the index $i$.
It is easy to see that the maps $\Xi_i$ are group homomorphisms.  
It will be useful for us to drop the index $i=2$. The reason for this will become apparent later when we observe that $O_2$ is $\Gamma_n$-invariant. We denote: $$\mathcal{V}_n= \bigoplus\limits_{i\in \{1\}\cup \{3,...,2^n-1\}} \mathbf{Z}_i$$ and define the following map:
$$\Xi:F_{2^n}^c\to \mathcal{V}_n\qquad \Xi(f)=(\Xi_1(f),\Xi_3(f),...,\Xi_{2^n-1}(f))$$
Note that since the map $\Xi:F_{2^n}^c\to \mathcal{V}_n$ is a homomorphism to an abelian group, it follows that for all $f,g\in F_{2^n}^c$ we have $\Xi(f^{-1}gf)=\Xi(g)$. A related technical lemma is the following:

\begin{lem}\label{Xiconjugacygeneral}
Let $g\in F_{2^n}^c$ and let $I\subset (0,1)$ be a closed interval with dyadic endpoints such that $Supp(g)\subset I$.
Let $\mathbf{S}^1\cong \mathbf{R}/\mathbf{Z}$ and let $f\in \textup{Homeo}^+(\mathbf{S^1})$ be such that:
\begin{enumerate}
\item $I\cdot f\subset (0,1)$.
\item $f\restriction I$ is a piecewise linear map with breakpoints (finitely many) that lie in $\mathbf{Z}[\frac{1}{2}]\cap I$
and slopes (whenever they exist) in $\{(2^n)^m\mid m\in \mathbf{Z}\}$.
\item $f$ maps $O_i\cap I$ to $O_i\cap (I\cdot f)$ for each $1\leq i\leq 2^n-1$.
\end{enumerate}
Then $\Xi(f^{-1} gf)=\Xi(g)$.
\end{lem}

\begin{proof}
Using the given conditions, we can find an element $h\in F_{2^n}^c$ such that $h\restriction I=f\restriction I$.
It follows that $\Xi(f^{-1} gf)=\Xi(h^{-1}gh)=\Xi(g)$ since $h\in F_{2^n}^c$. 
\end{proof}

Next, we let $\nu_i$ denote the generator of $\mathbf{Z}_i$ that corresponds to $1\in \mathbf{Z}$. We write elements of $\mathcal{V}_n$ as $\mathbf{Z}$-linear combinations of $\nu_1,\nu_3,...,\nu_{2^n-1}$.
For each $k\in \mathbf{N}$, we shall also denote by $\nu_k$ the element $\nu_l$ where $1\leq l\leq 2^n-1$ and $k\cong l\mod 2^n-1$.
This shall allow us to perform arithmetic modulo $2^n-1$ on the indices of $\nu_i$.
For notational completeness, we also include a symbol for $\nu_2$, but in our calculations we shall only be concerned with $\nu_k$ where the residue of $k\mod 2^n-1$ lies in $\{1,3,...,2^{n}-1\}$.

For $1\leq k\leq 2^n-2$, define $\zeta_k$ as a piecewise linear homeomorphism of $[0,1]$ whose support equals $(\frac{2}{2^{4n}},\frac{2+2^{n+1}k}{2^{4n}})$ and:
$$\zeta_k:=\begin{cases}
[\frac{2}{2^{4n}},\frac{2+k}{2^{4n}}]\mapsto [\frac{2}{2^{4n}},\frac{2+2^nk}{2^{4n}}]&\text{ linear with slope }2^n.\\
[\frac{2+k}{2^{4n}},\frac{2+2^nk}{2^{4n}}]\mapsto [\frac{2+2^nk}{2^{4n}},\frac{2+(2^{n+1}-1)k}{2^{4n}}]&\text{ linear with slope }1.\\
[\frac{2+2^nk}{2^{4n}},\frac{2+2^{n+1}k}{2^{4n}}]\mapsto [\frac{2+(2^{n+1}-1)k}{2^{4n}}, \frac{2+2^{n+1}k}{2^{4n}}]&\text{ linear with slope }\frac{1}{2^n}.\\
\end{cases}$$

We fix the notation $\Theta_n=\langle \zeta_1,...,\zeta_{2^n-2}, F_{2^n}'\rangle$ (note that $\zeta_2$ is included).
This group will play an important role in the proof of uniform simplicity.
The following lemma provides a key tool from this subsection.

\begin{lem}\label{weakbasis}
The following hold:
\begin{enumerate}
\item For each $1\leq k \leq 2^n-2$, $\Xi(\zeta_k)= -2\nu_{2+k}+\nu_{2+2k}$.

\item The elements $\Xi(\zeta_1),...,\Xi(\zeta_{2^n-2})$ freely generate a subgroup isomorphic to $\mathbf{Z}^{2^n-2}$. 

\item The kernel of $\Xi$ equals $F_{2^n}'$, and hence $\Xi$ is just the abelianization map: $$F_{2^n}^c\to F_{2^n}^c/[F_{2^n}^c,F_{2^n}^c]=F_{2^n}^c/F_{2^n}'$$

\item $\langle \Xi(\zeta_1),...,\Xi(\zeta_{2^n-2})\rangle$  is a finite index subgroup of $\Xi(F_{2^n}^c)$.

\item The group $\Theta_n=\langle \zeta_1,...,\zeta_{2^n-2},F_{2^n}'\rangle$ is a finite index subgroup of $F_{2^n}^c$.
And the set $\{\zeta_1,...,\zeta_{2^n-2}\}$ forms a weak generating set for $\Theta_n$.

\end{enumerate}
\end{lem}

\begin{proof}
{\bf Part $1$}: This follows from a direct calculation of $((\zeta_k)_-'(x),(\zeta_k)_+'(x))$ at the breakpoints $\frac{2}{2^{4n}}, \frac{2+k}{2^{4n}},\frac{2+2^nk}{2^{4n}}, \frac{2+2^{n+1}k}{2^{4n}}$ which are $(1, 2^n)$, $(2^n,1)$, $(1,2^{-n})$ and $(2^{-n},1)$ respectively, and since:
$$2+2^{n+1}k\cong 2+2k\mod 2^n-1\qquad 2+2^nk, 2+k\cong 2+k \mod 2^n-1$$
Note that this is well defined since for any $1\leq k \leq 2^n-2$: $$2+k,2+2k\not\cong 2 \mod 2^n-1$$
and the calculation does not consider the value $((\zeta_k)_-'(x),(\zeta_k)_+'(x))$ at $\frac{2}{2^{4n}}$ since $\frac{2}{2^{4n}}\in O_2$, and recall that $\mathcal{V}_n= \bigoplus\limits_{i\in \{1\}\cup \{3,...,2^n-1\}} \mathbf{Z}_i$.

{\bf Part $2$}: It is enough to pass to the natural homomorphism $\nu:\mathcal{V}_n\to (\mathbf{Z}/2\mathbf{Z})^{2^n-2}$, where we see that the images of $\nu(\Xi(\zeta_1)),...,\nu(\Xi(\zeta_{2^n-2}))$ are of the form $\Xi(\zeta_k)=(0,...,0,1,0,...,0)$ and freely generate a copy of $(\mathbf{Z}/2\mathbf{Z})^{2^n-2}$. Hence they also freely generate a copy of $(\mathbf{Z})^{2^n-2}$ in $\mathcal{V}_n$. 
(This follows from the elementary observation that if a nontrivial linear combination of vectors in $\mathbf{Z}^k$ with integer coefficients equals zero, then upon dividing the expression by the greatest common divisor of the coefficients, we may assume that at least one of the coefficients is odd in such a linear combination.) 

 
{\bf Part $3$}: First note that $\Xi$ being a map to an abelian group, must pass through the abelianization of $F_{2^n}^c$ which is isomorphic to $\mathbf{Z}^{2^n-2}$. 
Moreover, the image of $\Xi$ is a torsion-free abelian group and contains a subgroup $\langle \Xi(\zeta_1),...,\Xi(\zeta_{2^n-2})\rangle$, which is isomorphic to $\mathbf{Z}^{2^n-2}$. This means that in fact $\Xi$ is the abeliniazation map, or else $\Xi(F_{2^n}^c)$ would be a torsion-free abelian group of rank less than $2^n-2$ that contains such a subgroup, a contradiction.

{\bf Part $4$}: We have noted in the previous paragraph that $\langle \Xi(\zeta_1),...,\Xi(\zeta_{2^n-2})\rangle$, which is isomorphic to $\mathbf{Z}^{2^n-2}$,
is a subgroup of $\Xi(F_{2^n}^c)\cong \mathbf{Z}^{2^n-2}$. It follows that $\langle \Xi(\zeta_1),...,\Xi(\zeta_{2^n-2})\rangle$ is a finite index subgroup of $\Xi(F_{2^n}^c)$. 

{\bf Part $5$}: This follows from the previous part since $\Theta_n$ is the pre-image of $\langle \Xi(\zeta_1),...,\Xi(\zeta_{2^n-2})\rangle$ under the map $\Xi$, together with Lemma \ref{roots}.
\end{proof}
We remark that it is not clear to us what the index of $\Theta_n$ in $F_{2^n}^c$ is.

\section{A weak generating set for $\Delta_n$.}\label{SectionDeltan}

Recall that we identify $\mathbf{S}^1\cong \mathbf{R}/\mathbf{Z}$.
We define $(\Gamma_{n}')^c$ as the subgroup of $\Gamma_n'$ consisting of elements which fix $0$ and whose germs at $0$ are trivial.
In other words, there are precisely the elements of $\Gamma_n'$ whose support is contained in some compact sub-interval of $(0,1)$.
Also, note that $(\Gamma_{n}')^c=F_{2^n}^c\cap \Gamma_n'$.
We define $\Delta_n= (\Gamma_{n}')^c\cap \Theta_n$, where $\Theta_n$ is the group defined in the previous section. 
Since $\Theta_n$ is a finite index subgroup of $F_{2^n}^c$, and $(\Gamma_n')^c\leq F_{2^n}^c$, we obtain that $\Delta_n$ is a finite index subgroup of $(\Gamma_n')^c$.

First, we observe the following. Recall that in Proposition \ref{Fntrans} and in Section \ref{minimalweak}, we denoted the orbits of the groups $F_{2^n},F_{2^n}',F_{2^n}^c$ on $(0,1)\cap \mathbf{Z}[\frac{1}{2}]$ (which coincide) as $O_1,...,O_{2^n-1}$ and chose the transversal as $\{\frac{i}{2^n}\mid 1\leq i\leq 2^n-1\}$ so that $\frac{i}{2^n}\in O_i$.
In what follows, $\leq_{\text{ f.i. }}$ denotes a subgroup of finite index.
\begin{lem}\label{deltan}
The following hold:
\begin{enumerate}
\item $F_{2^n}'\leq \Delta_n\leq \Theta_n\leq_{\text{ f.i. }} F_{2^n}^c\leq F_{2^n}$ and $\Delta_n\leq_{\text{ f.i. }} (\Gamma_{n}')^c\leq F_{2^n}^c$.
\item $\Delta_n'=\Theta_n'=F_{2^n}'$.
\item The orbits of the action of $\Delta_n$ and $\Theta_n$ on $(0,1)\cap \mathbf{Z}[\frac{1}{2}]$ are $O_1,...,O_{2^n-1}$.
\end{enumerate}
\end{lem}

\begin{proof}
{\bf Part $1$}: The fact that $\Delta_n$ contains $F_{2^n}'$ as a subgroup
follows immediately by considering the image in $\Gamma_n'$ of the diagonal copy of $F_{2^n}'$ in $\Omega_n'$.
This image lies in both $(\Gamma_{n}')^c$ and $\Theta_n$.
We have already observed that $\Delta_n$ is a finite index subgroup of $(\Gamma_n')^c$.

{\bf Part $2$}: Since $F_{2^n}'\leq \Delta_n\leq F_{2^n}$, it holds that $F_{2^n}''\leq \Delta_n'\leq F_{2^n}'$. Using $F_{2^n}''=F_{2^n}'$, we conclude that $\Delta_n'=F_{2^n}'$. The same for $\Theta_n'$.

{\bf Part $3$}: These are precisely the orbits of the actions of $F_{2^n},F_{2^n}'$ on $(0,1)\cap \mathbf{Z}[\frac{1}{2}]$. And $F_{2^n}'\leq \Delta_n,\Theta_n\leq F_{2^n}$.
\end{proof}

We also refer to the natural lift of $\Delta_n$ in $\Omega_n'$ (contained in the analog of $F_{2^n}^c$ in $\Omega_n'$) as $\Delta_n$. Which group is being denoted by the notation will be made clear from the context. The same applies to the natural copies of $F_{2^n}',F_{2^n}^c,F_{2^n}$ in $\Gamma_n$ and $\Omega_n$. 
The goal of this section is to prove the following proposition, which is a key ingredient in the proof of uniform simplicity of $\Gamma_n'$. 

\begin{prop}\label{weakbasisc}
There exist elements $f,g\in \Gamma_n'$ such that:
\begin{enumerate}
\item The elements in $\{[\zeta_i^{-1},f]\mid 1\leq i\leq 2^n-2\}$ form a weak generating set for $\Delta_n$.
\item For each $1\leq i\leq 2^n-2$, we have $g^{-1}[\zeta_i^{-1},f]g\in F_{2^n}'=\Delta_n'$.
\end{enumerate}
\end{prop}

\subsection{The circle action of \texorpdfstring{$\{\Gamma_n'\}_{n\geq 2}$}.}

Now we study the orbits of the action of $\{\Gamma_n'\}_{n\geq 2}$ on the set of dyadics in the circle.
To avoid the ambiguity with the identification of $0,1$ in $\mathbf{R}/\mathbf{Z}$, we simply identify the circle (set-theoretically) with $[0,1)$ and study the orbits of the actions of our groups on the dyadics in this set.
Consider an element $f\in \Gamma_n'$ and a closed interval $I\subset (0,1)$ (which may be a single point) such that
$I\cup (I\cdot f)$ does not contain $0$.
To each such pair, we associate an integer, called the \emph{degree of the pair $(f,I)$} as follows.
Let $\tilde{f}$ be a lift of $f$ in $\Omega_n$ and view $I$ as the same interval in $[0,1)$ but now also as a sub-interval of $\mathbf{R}$. From our hypothesis, it follows that there is an $m\in \mathbf{Z}$ such that $I\cdot \tilde{f}\subset (m,m+1)$.
Then the degree of the pair $(f,I)$ is $k$ where $k$ is the residue of $m$ modulo $n$, i.e. $0\leq k\leq n-1$ and $k\cong m\mod n$. 
It is clear from the definition that the degree is independent of the choice of lift (since all lifts differ by translations that are integer multiples of $n$). An informal way of understanding the degree of the pair $(f,I)$ using the action on the circle is precisely how many times $I$ crosses $0$ counter-clockwise in $\mathbf{R}/\mathbf{Z}$ when the element $f$ is applied to it, modulo $n$.

\begin{lem}\label{degreelem}
Let $I\subset (0,1)$ be a closed interval, $k\in \{0,1,...,n-1\}$ and $U\subset (0,1)$ a nonempty open set. 
Then there is an element $f\in \Gamma_n'$ such that $I\cdot f\subset U$ and the degree of the pair $(f,I)$ is $k$.
\end{lem}

\begin{proof}
Let $U_1\subset (k,k+1)\subset \mathbf{R}$ be an open set whose image under the map $\mathbf{R}\to \mathbf{R}/\mathbf{Z}$ is $U$.
Using Proposition \ref{periodproximal}, we find $f\in\Omega_n'$ such that $I\cdot f\subset U_1$. It follows that the image of $f$ in $\Gamma_n'$ has the required property.
\end{proof}








The action of the groups $F_{2^n},F_{2^n}^c,F_{2^n}'$ on $[0,1)\cap \mathbf{Z}[\frac{1}{2}]$ have the orbits $\{0\},O_1,...,O_{2^n-1}$.
Since the orbits $O_1,...,O_{2^n-1}$ were originally defined for the action of these groups on $(0,1)\cap \mathbf{Z}[\frac{1}{2}]$,
this did not include the point $0$. It is natural to include the point $0$ in the set $O_{2^n-1}$ since $0\cong 2^n-1\mod 2^n-1$, and since the image of $0$ under a special element lies in $O_{2^n-1}$ thanks to Lemma \ref{GammaOrbits}.
For the rest of the article, we shall do so whenever considering the action on $[0,1)$.
For each $k\in \mathbf{N}$, we shall denote by $O_k$ the orbit $O_l$ where $1\leq l\leq 2^n-1$ is such that $l\cong k\mod 2^n-1$.

Since $F_{2^n}'\leq \Gamma_n'$, each orbit of the action of $\Gamma_n'$ on $[0,1)\cap \mathbf{Z}[\frac{1}{2}]$ is a union of some orbits of the action of $F_{2^n}'$ on $[0,1)\cap \mathbf{Z}[\frac{1}{2}]$. 
It follows that there is a partition of $\{1,...,2^n-1\}$ consisting of sets 
$\chi_0,\chi_1,...,\chi_{\eta_n}$ for some $\eta_n\in \mathbf{N}$ such that the orbits of the action of $\Gamma_n'$ on $[0,1)\cap \mathbf{Z}[\frac{1}{2}]$
are precisely: $$\Phi_i=\bigcup_{j\in \chi_i} O_j\qquad 0\leq i\leq \eta_n$$
We remark that the precise calculation of the value of $\eta_n$, while an interesting curiosity by itself, will not be relevant to our arguments
which do not depend on knowledge of the precise value. 
(It turns out that when $n$ is an odd prime, $\eta_n=\frac{2^n-2}{n}$.
This will be left as an exercise for the reader, and the proofs in this section will provide the necessary hints).
However, in Lemma \ref{forwardorbits} we shall provide a description of the sets $\chi_0,...,\chi_{\eta_n}$, and hence of the orbits $\Phi_i=\bigcup_{j\in \chi_i} O_j$ for each $0\leq i\leq \eta_n$.

We modify the special element $\tau$ constructed in Subsection \ref{specialelementsconstruction}. 
Let $I:=[-\frac{1}{2^n}, \frac{1}{2^{n-1}}]$. We had defined $\lambda:I\to I$ as follows:
$$\lambda:=\begin{cases}
[-\frac{1}{2^n},\frac{1-2^n}{2^{2n}}]\mapsto [-\frac{1}{2^n},0] & \text{ linear with slope } 2^n.\\
[\frac{1-2^n}{2^{2n}},0]\mapsto [0, \frac{2^{n}-1}{2^{2n-1}}]&\text{ linear with slope }2.\\
[0,\frac{1}{2^{n-1}}]\mapsto [\frac{2^{n}-1}{2^{2n-1}},\frac{1}{2^{n-1}}]
&\text{ linear with slope }\frac{1}{2^n}.
\end{cases}$$
We define the $1$-periodic homeomorphism, and a special element, as $g:\mathbf{R}\to \mathbf{R}$ satisfying $\Supp(g)=\Int(I)+\mathbf{Z}$ and $g\restriction I=\lambda$. 
Let $J\subset (0,1)$ be a nonempty open interval such that $(J+\mathbf{Z})\cap  (\Int(I)+\mathbf{Z})=\emptyset$.
Using the fact that $\Omega_n$ acts in a $1$-periodically proximal fashion on $\mathbf{R}$, we find $f\in \Omega_n$ such that 
$(\Int(I)+\mathbf{Z})\cdot f\subset J+\mathbf{Z}$.
Then we declare $\tau=gf^{-1} g^{-1} f$.
Note that $\tau\restriction I+\mathbf{Z}=g\restriction I+\mathbf{Z}$ and that $\tau$ is a special element in $\Omega_n'$.
\begin{lem}\label{orbitscircleaction}
Let $\tau\in \Omega_n'$ be the special element as above and let $\tau_1$ be its image in $\Gamma_n'$.
Let $x,y\in (0,1)$ be such that $0\cdot \tau=x$ and $y\cdot \tau=1$.
Then the following hold:
\begin{enumerate}

\item  $y\cdot \tau_1=0$ and $0\cdot \tau_1\in O_{2^n-1}$.

\item The restriction $\tau_1\restriction [y, 1)\to [0,x)$ maps $O_2\cap [y, 1)$ to $O_2\cap [0,x)$.

\item The restriction $\tau_1\restriction [y, 1)\to [0,x)$ maps $O_i\cap [y,1)$ to $O_{2(i-1)}\cap [0,x)$. 

\item The restriction $g\restriction [0,y)\to [x, 1)$ maps $O_i\cap [0,y)$ to $O_i\cap [x,1)$ for each $1\leq i\leq 2^n-1$. 

\item $O_2$ is $\tau_1$-invariant.

\end{enumerate}
\end{lem}

\begin{proof}
For the first part, $y\cdot \tau_1=0$ is clear from the definition (remember that $\tau_1$ acts on $[0,1)$) and $0\cdot \tau_1\in O_{2^n-1}$ follows from an application of Lemma \ref{GammaOrbits}.
To see the second and third parts, first observe that $y\in O_1$ since $y=1+\frac{1-2^n}{2^{2n}}$.
Moreover, an element of the form $y+\frac{i}{2^n}\in (y, 1)$ is mapped to $\frac{2i}{2^n}$ since the slope in the interior of this interval is $2$ and $y$ is mapped to $0$.
Combining these two facts shows that $\tau_1\restriction (y, 1)\to (0,x)$ maps $O_i\cap (y,1)$ to $O_{2(i-1)}\cap (0,x)$.
This settles the claim in parts $(2)$ and $(3)$.
The second last part follows from an immediate application of Lemma \ref{GammaOrbits}, and the last part follows from parts $(2),(4)$.
\end{proof}

The following lemma provides a description of the sets $\chi_0, \chi_1,...,\chi_{\eta_n}$, and hence of the orbits $\Phi_i=\bigcup_{j\in \chi_i} O_j$ for each $0\leq i\leq \eta_n$.

\begin{lem}\label{forwardorbits}
The map $i\mapsto 2(i-1)\mod 2^n-1$ induces a bijection of the set $\{1,...,2^n-1\}$ with fixed point $2$.
The orbits of the forward and backward iterates of this map are precisely the sets $\chi_0,...,\chi_{\eta_n}$.
Moreover, $O_2$ is $\Gamma_n$-invariant and $\Gamma_n'$-invariant.
\end{lem}

\begin{proof}
The first part is an exercise in elementary number theory. 
We show the second part. First note that $\Gamma_n=\langle \tau_1,F_{2^n}\rangle$ where $\tau_1$ is the element from Lemma \ref{orbitscircleaction}.
Next, note that $F_{2^n}$ preserves $O_1,...,O_{2^n-1}$, and that $\tau_1$ preserves $O_2$ (also from Lemma \ref{orbitscircleaction}).
Then both claims follow from these upon applying Lemma \ref{orbitscircleaction}. 
\end{proof}
We recall for the reader that we had shown in Lemma \ref{gammaorbits} that the orbits of $\Gamma_n$ and $\Gamma_n'$ on $[0,1)\cap \mathbf{Z}[\frac{1}{2}]$ coincide.
Since $O_2$ is $\Gamma_n$-invariant and $\Gamma_n'$-invariant thanks to the above, we shall declare $\Phi_0=O_2$ and $\chi_0=\{2\}$ for the rest of the article. 
Another consequence is the following.

\begin{cor}
Let $f,g\in \Gamma_n'$ and $x\in (0,1)\cap \mathbf{Z}[\frac{1}{2}]$ be such that the degrees of $(f,x),(g,x)$ are equal.
Then there is a $1\leq i\leq 2^n-1$ such that $x\cdot f,x\cdot g\in O_i$. 
\end{cor}

\begin{proof}
An application of Lemma \ref{GammaGen} shows that $\Gamma_n'$ is generated by $F_{2^n}\cap \Gamma_n'$ together with the special element $\tau_1$ from above.
The subgroup $F_{2^n}\cap \Gamma_n'$ preserves $O_1,...,O_{2^n-1}$ and the element $\tau_1$ acts according to the prescribed conditions from Lemma \ref{orbitscircleaction}. Applying these in an inductive fashion, we immediately obtain the required condition. 
\end{proof}

\begin{lem}\label{degree0}
Let $f\in \Gamma_n'$ and $I\subset (0,1)$ be a closed interval such that the pair $(f,I)$ has degree $0$.
Let $g\in F_{2^n}^c$ be such that $Supp(g)\subseteq I$. Then $g^{-1}f^{-1}gf\in F_{2^n}'$.
\end{lem}

\begin{proof}
Recall the map $\Xi:F_{2^n}^c\to \oplus_{i\in \{1,3,...,2^n-1\}}\mathbf{Z}_i$ defined in Section \ref{minimalweak},
and that $ker(\Xi)=F_{2^n}'$ from Lemma \ref{weakbasis}. Since $f,g$ satisfy the hypothesis of Lemma \ref{Xiconjugacygeneral} we know that $\Xi(f^{-1}gf)=\Xi(g)$. We observe that
$g^{-1}f^{-1}gf\in F_{2^n}^c$ and that: $$\Xi(g^{-1}f^{-1}gf)=\Xi(g^{-1})+\Xi(f^{-1}gf)=\Xi(g^{-1})+\Xi(g)=0$$
It follows that $g^{-1}f^{-1}gf\in F_{2^n}'$.
\end{proof}


Recall the construction of the elements $\{\zeta_1,...,\zeta_{2^n-2}\}$ from Section \ref{minimalweak}.
In particular, recall that for each $1\leq i\leq 2^n-2$, $Supp(\zeta_i)=(\frac{2}{2^{4n}},\frac{2+2^{n+1}k}{2^{4n}})$.
We fix the notation $\mathcal{I}_n=[\frac{2}{2^{4n}},\frac{2+2^{2n+1}}{2^{4n}}]$.
Note that for each $1\leq i\leq 2^n-2$ we have $Supp(\zeta_i)\subset \mathcal{I}_n\subset (0,1)$.

\begin{lem}\label{perm1}
Let $f\in \Gamma_n'$ be such that the degree of $(f,\mathcal{I}_n)$ equals $1$.
Then $\Xi(f^{-1}\zeta_if)=\Xi(\zeta_{2i})$.
\end{lem}

\begin{proof}
This follows from an application of Lemmas \ref{weakbasis} (part $1$) and \ref{orbitscircleaction}.
\end{proof}



\subsection{The maps \texorpdfstring{$\gimel,\gimel_1,...,\gimel_{\eta_n}$}.}
Recall that there is a partition of $\{1,...,2^n-1\}$ consisting of sets 
$\chi_0,...,\chi_{\eta_n}$ for some $\eta_n\in \mathbf{N}$ such that the orbits of the action of $\Gamma_n'$ on $[0,1)\cap \mathbf{Z}[\frac{1}{2}]$
are precisely: $$\Phi_i=\bigcup_{j\in \chi_i} O_j\qquad 0\leq i\leq \eta_n$$
Also recall that $O_2$ is $\Gamma_n$-invariant thanks to the Lemma \ref{forwardorbits}, and that we had declared $\Phi_0=O_2$ and $\chi_0=\{2\}$. 
The aforementioned orbit structure allows us to define the maps (dropping the index $0$ intentionally) for $1\leq i\leq \eta_n$:
$$\gimel_i:\Gamma_n\to \mathbf{Z}_i \qquad \gimel_i(f)=\sum_{x\in \Phi_i}  D_f(x)\qquad D_f(x)=\log_{2^n}f_+^{'}(x) -\log_{2^n}f_-^{'}(x)$$ 
$$\gimel:\Gamma_n\to \oplus_{1\leq i\leq \eta_n} \mathbf{Z}_i\qquad \gimel(f)=(\gimel_1(f),...,\gimel_{\eta_n}(f))$$
where each $\mathbf{Z}_i$ is an indexed copy of $\mathbf{Z}$.
Note that while the left and right derivatives of an element at a point $x\in (0,1)$ are defined as before, the left and right derivatives at $0$ are also defined in the natural way using the orientation on the circle.
It follows from the definition that the maps $\gimel$ and $\gimel_i$ are group homomorphisms.
The following is an obvious consequence of the simplicity of $\Gamma_n'$ (Proposition \ref{SimpleProp}).

\begin{lem}\label{gimelvanish}
The homomorphisms $\gimel,\gimel_1,...,\gimel_{\eta_n}$ vanish on $\Gamma_n'$.
\end{lem}

We consider the restriction of $\gimel$ to the subgroup $F_{2^n}^c$ of $\Gamma_n$.
Recall that $\Theta_n=\langle \zeta_1,...,\zeta_{2^n-2}, F_{2^n}'\rangle$,
and recall from Lemma \ref{weakbasis} that $\{\zeta_1,...,\zeta_{2^n-2}\}$ forms a weak generating set for $\Theta_n$ and that
$\{\Xi(\zeta_k)\mid 1\leq k\leq 2^n-2\}$ is a free basis for the abelianization $\mathbf{Z}^{2^n-2}$ of $\Theta_n$. 
Note that the maps $\Xi_i,\Xi$ were defined for the action of $F_{2^n}$ on $(0,1)$ but are also defined in the same way for the action on $[0,1)$ that we are considering (and indeed coincide with the previous definition).

\begin{lem}\label{gimel1}
There is a homomorphism $\varsigma: \Xi(\Theta_n)\to \oplus_{1\leq i\leq \eta_n} \mathbf{Z}_i$ such that
for each $f\in F_{2^n}^c$, $\gimel(f)=\varsigma(\Xi(f))$.
In particular, we have the following commutative diagram:
\[
\begin{tikzcd}
 & \Xi(\Theta_n) \arrow[dr,"\varsigma"] \\
\Theta_n \arrow[ur,"\Xi"] \arrow[rr,"\gimel"] && \oplus_{1\leq i\leq \eta_n} \mathbf{Z}_i
\end{tikzcd}
\]
Moreover, $\varsigma(\Xi(\zeta_k))\in \oplus_{1\leq i\leq \eta_n}\mathbf{Z}_i$ is of the form $(0,...,-1,...,0)$, where the sole nonzero entry $-1$ is in the $j$-th coordinate $\mathbf{Z}_j$ where $2+k,2+2k\in \chi_j$.
\end{lem}

\begin{proof}
Recall from Lemma \ref{weakbasis} that $\Xi:\Theta_n\to \mathbf{Z}^{2^n-2}$ has kernel $\Theta_n'=F_{2^n}'$ and hence is the abelianization map. Moreover, $\zeta_1,...,\zeta_{2^n-2}$ forms a weak generating set for the group $\Theta_n$.
So in particular, $\Xi(\Theta_n)=\langle \{\Xi(\zeta_k)\mid 1\leq k\leq 2^n-2\}\rangle$.
Any homomorphism of a group to an abelian group must pass through the abelianization,
and so the existence of the map $\varsigma$ follows immediately. 
Recall that each orbit $\Phi_i$ is a union of some of the elements in $\{O_1, ... ,O_{2^n-1}\}$,
and that $\Xi(\zeta_k)=-2\nu_{2+k}+\nu_{2+2k}$. 
It follows from Lemma \ref{forwardorbits} that $2+k,2+2k$ lie in $\chi_j$ for some $1\leq j\leq \eta_n$, and so $O_{2+k},O_{2+2k}$ are subsets the same $\Gamma_n'$-orbit.
So our second claim from the statement of lemma follows from the definition of the map $\gimel$.
\end{proof}




We denote the kernel of the restriction map $\gimel:\Theta_n\to \oplus_{1\leq i\leq \eta_n}\mathbf{Z}$ as $\Pi_n$.
We define $\Psi_n$ as the kernel of $\varsigma: \Xi(\Theta_n)\to \oplus_{1\leq i\leq \eta_n} \mathbf{Z}_i$. 

\begin{lem}\label{deltapi1}
$\Delta_n\leq \Pi_n$.
\end{lem}

\begin{proof}
Since $\Delta_n= (\Gamma_n')^c\cap \Theta_n$, and since $\gimel$ vanishes on $\Gamma_n'$ (see Lemma \ref{gimelvanish}), we obtain 
$\Delta_n\leq \Pi_n$.
\end{proof}

\begin{lem}\label{pipsi1}
$\Pi_n=\Xi^{-1}(\Psi_n)$.
\end{lem}

\begin{proof}
This follows from the commutative diagram in Lemma \ref{gimel1}, together with the fact that $F_{2^n}'\subseteq \Pi_n$.
\end{proof}

\begin{lem}\label{psigenset}
$\Psi_n$ admits a generating set of the form: $$\{\Xi(\zeta_i)-\Xi(\zeta_{2i})\mid 1\leq i\leq 2^n-2\}$$
\end{lem}

\begin{proof}
From Lemma \ref{gimel1} we know that $\varsigma(\Xi(\zeta_k))\in \oplus_{1\leq i\leq \eta_n}\mathbf{Z}_i$ is of the form $(0,...,-1,...,0)$, where the sole nonzero entry $-1$ is in the $j$-th coordinate $\mathbf{Z}_j$ where $2+k\in \chi_j$.
It follows that $\Psi_n$ consists of elements of the form $\sum_{1\leq i\leq 2^n-2} \alpha_i \Xi(\zeta_i)$ where $\alpha_1,...,\alpha_{2^n-2}\in \mathbf{Z}$ are coefficients satisfying that for each $1\leq j\leq \eta_n$: $$\sum_{2+l\in \chi_j}\alpha_{l}=0$$
Therefore, it follows that $\Psi_n$ admits a generating set of the form: 
$$\{\Xi(\zeta_i)-\Xi(\zeta_{j})\mid 1\leq i,j\leq 2^n-2\text{ such that }2+i,2+j\in \chi_l\text{ for some }1\leq l\leq \eta_n\}$$
Moreover, from Lemma \ref{forwardorbits} we know that if $2+i\in \chi_l$, then $\chi_l=\{2+i,2+2i,2+2^2i,...,2+2^ki,...\}$.
It follows that $\Psi_n$ admits a generating set of the form: $$\{\Xi(\zeta_i)-\Xi(\zeta_{2i})\mid 1\leq i\leq 2^n-2\}$$
\end{proof}


\begin{lem}\label{weakbasispi}
Let $f\in \Gamma_n'$ be any element such that $(f,\mathcal{I}_n)$ has degree $1$.
Then $\Pi_n$ admits a weak generating set consisting of elements of the form: 
$$\{\zeta_{k}f^{-1}\zeta_{k}^{-1}f\mid 1\leq k\leq 2^n-2\}$$
Moreover, for each $1\leq k\leq 2^n-2$, we have that $\zeta_{k}f^{-1}\zeta_{k}^{-1}f\in \Delta_n$.
\end{lem}

\begin{proof}
Let $P=\{\zeta_{k}f^{-1}\zeta_{k}^{-1}f\mid 1\leq k\leq 2^n-2\}$.
It suffices to show that $\Pi_n=\langle \Pi_n', P\rangle$.
Since $F_{2^n}'\leq \Pi_n\leq F_{2^n}$, from Theorem \ref{abelianization} we have $\Pi_n'=F_{2^n}'$. So in fact it suffices to show that $\Pi_n=\langle F_{2^n}', P\rangle$.
The kernel of the map $\Xi\restriction \Pi_n:\Pi_n\to \Xi(\Pi_n)$ equals $F_{2^n}'$, and hence $\Pi_n$ is generated by $F_{2^n}'$ and $\Xi^{-1}(L)$ where $L$ is any generating set for $\Xi(\Pi_n)$.
From Lemma \ref{pipsi1}, we know that $\Xi(\Pi_n)=\Psi_n$.
From Lemma \ref{psigenset}, we know that $\Psi_n$ admits a generating set of the form: $$\{\Xi(\zeta_i)-\Xi(\zeta_{2i})\mid 1\leq i\leq 2^n-2\}$$
From Lemma \ref{perm1}, using the fact that $(f,\mathcal{I}_n)$ has degree $1$, we know that: $$\Xi(\zeta_{k}f^{-1}\zeta_{k}^{-1}f)=\Xi(\zeta_{k})+\Xi(f^{-1}\zeta_{k}^{-1}f)=\Xi(\zeta_i)+\Xi(\zeta_{2i}^{-1})=\Xi(\zeta_i)-\Xi(\zeta_{2i})$$
This means that $\Pi_n$ admits: $$P=\{\zeta_{k}f^{-1}\zeta_{k}^{-1}f\mid 1\leq k\leq 2^n-2\}$$ as a weak generating set. 

Recall that $\Delta_n=\Theta_n\cap (\Gamma_n')^c$.
We know that for each $1\leq k\leq 2^n-2$, we have that $\zeta_{k}f^{-1}\zeta_{k}^{-1}f\in (\Gamma_n')^c$.
We simply need to show that for each $1\leq k\leq 2^n-2$, we have that $\zeta_{k}f^{-1}\zeta_{k}^{-1}f\in \Theta_n$.
Note that $\zeta_{k}\in \Theta_n$.
So we only need to show that $f^{-1}\zeta_{k}^{-1}f\in \Theta_n$.
Since from Lemma \ref{perm1}, $\Xi(f^{-1}\zeta_{k}^{-1}f)=\Xi(\zeta_{2k}^{-1})$, we have
$\Xi(f^{-1}\zeta_{k}^{-1}f)+\Xi(\zeta_{2k})=0$ and so: $$f^{-1}\zeta_{k}^{-1}f\zeta_{2k}\in ker(\Xi)=F_{2^n}'$$
Since $\zeta_{2k}\in \Theta_n, F_{2^n}'\leq \Theta_n$, it follows that $f^{-1}\zeta_{k}^{-1}f\in \Theta_n$.
\end{proof}

\begin{lem}\label{same}
$\Pi_n=\Delta_n$.
\end{lem}

\begin{proof}
First note that $\Delta_n\leq \Pi_n$ from Lemma \ref{deltapi1}.
From Lemma \ref{weakbasispi} we know that for any $f\in \Gamma_n'$ so that $(f,\mathcal{I}_n)$ has degree $1$, $\Pi_n$ admits a weak generating set consisting of elements of the form
$\{\zeta_{k}f^{-1}\zeta_{k}^{-1}f\mid 1\leq k\leq 2^n-2\}$.
Indeed, since all of these elements lie in $\Delta_n$ (also from Lemma \ref{weakbasispi}), and since $\Pi_n'=F_{2^n}'\leq \Delta_n$, we obtain that $\Pi_n\leq \Delta_n$ which finishes the proof.
\end{proof}

\begin{cor}\label{weakbasisdelta}
Let $f\in \Gamma_n'$ be such that $(f,\mathcal{I}_n)$ has degree $1$.
Then the set of elements $\{\zeta_{k}f^{-1}\zeta_{k}^{-1}f\mid 1\leq k\leq 2^n-2\}$ is a weak generating set for $\Delta_n$. 
\end{cor}

\begin{proof}
From Lemma \ref{weakbasispi} we know that $\Pi_n$ admits such a weak generating set, and from Lemma \ref{same} we know that $\Pi_n=\Delta_n$.
This completes the proof.
\end{proof}

Now we are ready to prove the main Proposition of this section.

\begin{proof}[Proof of Proposition \ref{weakbasisc}]
Using Lemma \ref{specialelementsgalore1}, we find a special element $g\in \Omega_n'$ such that: $$Supp(g)\cap (\mathcal{I}_n+\mathbf{Z})=\emptyset$$
We denote the image of this special element in $\Gamma_n'$ as also $g$.
Note that from our assumption it follows that $Supp(g)\cap \mathcal{I}_n=\emptyset$ and $[g,\zeta_i]=1$ for each $1\leq i\leq 2^n-2$.
Let $J\subset (0,1)$ be a nonempty open interval such that $(g,J)$ has degree $1$. 
Using Lemma \ref{degreelem}, we find an $f\in \Gamma_n'$ such that:
\begin{enumerate}
\item $\mathcal{I}_n\cdot f\subset J\cdot g$.
\item $(f,\mathcal{I}_n)$ has degree $1$.
\end{enumerate}
From Corollary \ref{weakbasisdelta}, it follows that: $$\{\zeta_if^{-1}\zeta_i^{-1}f\mid 1\leq i\leq 2^n-2\}$$ is a weak generating for $\Delta_n$.
Recall that for each $1\leq i\leq 2^n-2$ we have $Supp(\zeta_i)\subset \mathcal{I}_n\subset (0,1)$.
From our choices of $f$ and $g$, it follows that the pair $(fg^{-1},\mathcal{I}_n)$ has degree $0$.
We obtain:
$$g(\zeta_if^{-1}\zeta_i^{-1}f)g^{-1}=(g\zeta_ig^{-1})(gf^{-1} \zeta_i^{-1} fg^{-1})
=\zeta_i((fg^{-1})^{-1} \zeta_i^{-1} (fg^{-1}))$$
(Recall that $g, \zeta_i$ commute since they have disjoint supports).
Since the pair $(fg^{-1},\mathcal{I}_n)$ has degree $0$, from Lemma \ref{degree0} it follows that $\zeta_i((fg^{-1})^{-1} \zeta_i^{-1} (fg^{-1}))\in F_{2^n}'$, and so for each $1\leq i\leq 2^n-2$, we have that $\zeta_if^{-1}\zeta_i^{-1}f$ is conjugate to an element of $F_{2^n}'$ via an element of $\Gamma_n'$.
This completes the proof.
\end{proof}

\section{Uniform simplicity for \texorpdfstring{$\{\Gamma_n'\}_{n\geq 2}$}.}
The goal of this section is to prove the following:

\begin{thm}\label{unisim}
For each $n\geq 2$, the group $\Gamma_n'$ is uniformly simple.
\end{thm}
We will show this by showing that the hypothesis of Lemma \ref{BC} applies to the $4$-tuple: $$G=\Gamma_n'\qquad H_1=F_{2^n}'\qquad H_2=(\Gamma_n')^c\qquad H_3=\Delta_n$$
We will prove a series of technical lemmas that will lead to the proof of this theorem.
We encourage the reader to recall the content of Subsection \ref{BCsubsec} and also the previous section.
We also note that we consider conjugacy in the groups $\Gamma_n'$ and $\Omega_n'$ (rather than $\Gamma_n$ and $\Omega_n$).

\begin{lem}\label{unisimlem1}
Let $f\in \Gamma_n'\setminus \{id\}$. Then there are elements $\alpha_1,\alpha_2\in \Gamma_n'$ such that $[[f,\alpha_1],\alpha_2]\in F_{2^n}'\setminus \{id\}$.
\end{lem}

\begin{proof}
Since $f$ is non-trivial, we can find a non-empty open interval $I$ such that: $$0\notin \overline{(I\cdot f)\cup I}\qquad I\cap (I\cdot f)=\emptyset$$
We can find $\alpha_1\in F_{2^n}'\setminus \{id\}$ whose support is contained in $I$ so that: 
$$0\notin \overline{Supp(\alpha_1)\cup (Supp(\alpha_1)\cdot f)}\qquad Supp(\alpha_1)\cap (Supp(\alpha_1)\cdot f)=\emptyset$$
Note that the latter condition implies that $[f,\alpha_1]\neq id$, and the former implies that $[f,\alpha_1]\in (\Gamma_n')^c$. And using Lemma \ref{deltan}, we observe that $[f,\alpha_1]\in F_{2^n}$.
Using the elementary fact that the centralizer of $F_{2^n}'$ in $F_{2^n}$ is trivial, we choose $\alpha_2\in F_{2^n}'$ such that 
$[[f,\alpha_1],\alpha_2]\in F_{2^n}'\setminus \{id\}$.
\end{proof}

Recall that $t_n:\mathbf{R}\to \mathbf{R}$ is the map $x\cdot t_n=x+n$.

\begin{lem}\label{unisimlem2}
Let $n\geq 2$.
For each $g\in \Omega_n'\setminus \{id\}$, there exist 
$k\leq 2n+4$, $g_1,...,g_k\in C_{F_{2^n}'}(\Omega_n')$ and $h\in \Omega_n'$ such that 
$h$ pointwise fixes a neighborhood of $0$ and $g=hg_1...g_kt_n^{-l}$ for some $l\in \mathbf{Z}$.
\end{lem}

\begin{proof}
Let $l\in \mathbf{Z}$ be such that $x=0\cdot gt_n^l \in [-p,-p+1)$ for some $-p\in \{-n,...,-1\}$.
Using Lemma \ref{specialelementsgalore1}, find a special element $\alpha\in \Omega_n'$ which also lies in $C_{F_{2^n}'}(\Omega_n')$.
We first choose $\lambda_1\in F_{2^n}'$ such that $x\cdot \lambda_1\alpha\in  (-p+1,-p+2)$.
Continuing in this way, we choose $\lambda_1,...,\lambda_{p}\in F_{2^n}'$ such that $y=x\cdot \lambda_1\alpha...\lambda_{p}\alpha\in (0,1)$.
Since $y\in (0,1)$ is the image of $0$ upon applying an element of $\Omega_n'$, by Lemma \ref{GammaOrbits} it is in the $F_{2^n}'$ orbit of $0\cdot \alpha$.
Then we choose $\lambda_{p+1}\in F_{2^n}'$ such that $y\cdot \lambda_{p+1}=0\cdot \alpha$, which gives us $y\cdot \lambda_{p+1}\alpha^{-1}=0$, and so
$$0\cdot g t_n^{l}\lambda_1\alpha...\lambda_{p}\alpha\lambda_{p+1}\alpha^{-1}=0$$
Using Lemma \ref{zerofix}, we find elements $\beta_1,\beta_2\in \Omega_n'$ such that $\beta_1,\beta_2\in C_{F_{2^n}'}(\Omega_n')$ and
the element $h=g t_n^{l}\lambda_1\alpha...\lambda_{p}\alpha \lambda_{p+1}\alpha^{-1}\beta_1\beta_2$ pointwise fixes a neighborhood of $0$.
Then the expression $$g=h\beta_2^{-1}\beta_1^{-1} \alpha\lambda_{p+1}^{-1}\alpha^{-1}\lambda_p^{-1}...\alpha^{-1}\lambda_1^{-1}t^{-l}_n$$
has the desired form since $\lambda_1,...,\lambda_{p+1}, \beta_1,\beta_2,\alpha\in C_{F_{2^n}'}(\Omega_n')$.
\end{proof}

\begin{lem}\label{unisimlem3}
Let $n\geq 2$.
For each $g\in \Gamma_n'\setminus \{id\}$, there exist 
$k\leq 2n+4$, $g_1,...,g_k\in C_{F_{2^n}'}(\Gamma_n')$ and $h\in (\Gamma_n')^c$ such that 
$g=h(g_1...g_k)$.
\end{lem}

\begin{proof}
We choose a lift $\tilde{g}$ of $g$ and apply Lemma \ref{unisimlem2} to find 
$k\leq 2n+4$, $g_1,...,g_k\in C_{F_{2^n}'}(\Omega_n')$ and $h\in \Omega_n'$ such that 
$h$ pointwise fixes a neighborhood of $0$ and $\tilde{g}=hg_k^{-1}...g_1^{-1}t_n^{-l}$ for some $l\in \mathbf{Z}$.
The required expression is then obtained from applying the homomorphism $\Omega_n'\to \Gamma_n'$ to this expression.
\end{proof}

\begin{lem}\label{unisimlem4}
In the group $\Gamma_n'$, the subgroup $(\Gamma_n')^c$ contains a finite index subgroup $\Delta_n$ that is boundedly covered by the subgroup $F_{2^n}'$.
\end{lem}

\begin{proof}
Recall that $\Delta_n$ here is the group defined in Section \ref{SectionDeltan}.
Using Proposition \ref{weakbasisc}, we find elements $f,g\in \Gamma_n'$ such that
the elements in $\{[\zeta_i^{-1},f]\mid 1\leq i\leq 2^n-2\}$ form a weak generating set for $\Delta_n$ and they are all conjugate to $F_{2^n}'$ by $g$. Moreover, in Lemma \ref{deltan} we showed that $\Delta_n'=F_{2^n}'$ and that $\Delta_n$ is a finite index subgroup of $(\Gamma_n')^c$. It follows from Lemma \ref{BC1} that $\Delta_n$ is boundedly covered by $F_{2^n}'$ in $\Gamma_n'$. Since $\Delta_n$ is finite index in $(\Gamma_n')^c$, we are done.
\end{proof}

\begin{proof}[Proof of Theorem \ref{unisim}]
We show that the hypothesis of Lemma \ref{BC} applies to the $4$-tuple: $$G=\Gamma_n'\qquad H_1=F_{2^n}'\qquad H_2=(\Gamma_n')^c\qquad H_3=\Delta_n$$
(Recall that $F_{2^n}'\leq \Delta_n\leq_{f.i.} (\Gamma_n')^c\leq \Gamma_n'$.)
First, thanks to Proposition \ref{SimpleProp}, $\Gamma_n'$ is simple.
The first part of the hypothesis of Lemma \ref{BC} follows since the group $F_{2^n}'$ is $6$-uniformly simple thanks to Theorem \ref{commwid}.
The second part is the content of Lemma \ref{unisimlem1}, the third part is the content of Lemma \ref{unisimlem4}, and the fourth part is the content of Lemma \ref{unisimlem3}.
\end{proof}

\section{Lower bounds on $\{\mathcal{R}(\Gamma_n')\}_{n\geq 2}$ and $\{\mathcal{P}(\Gamma_n')\}_{n\geq 2}$.}

The main goal of this section is to prove that $\mathcal{R}(\Gamma_n')\geq \floor{\frac{n}{2}}$. Recall that $t_n\in \textup{Homeo}^+(\R)$ is the element $x\cdot t_n=x+n$.

\begin{prop}\label{mainproplb}
There exist elements $f,g\in \Omega_n'\setminus \Lambda_n$ such that whenever: $$g=g_1...g_kt_n^m\qquad \text{for some } g_1,...,g_k\in C_f(\Omega_n')\cup C_{f^{-1}}(\Omega_n'), m\in \mathbf{Z}$$ then $k\geq \floor{\frac{n}{2}}$.
\end{prop}

\begin{proof}

Let $f\in F_{2^n}'\setminus \{id\}$ be some fixed element, and using minimality of the action of $\Omega_n'$ on $\mathbf{R}$ (a consequence of Proposition \ref{periodproximal}), choose $g\in \Omega_n'$ such that $0\cdot g\in (\floor{\frac{n}{2}},\floor{\frac{n}{2}}+1)$. Clearly, $f,g\notin\Lambda$.
We will show that these elements witness the required phenomenon.

{\bf Claim}: For any $h\in C_{f}(\Omega_n')\cup C_{f^{-1}}(\Omega_n')$ and any $x\in \mathbf{R}$, $|x-x\cdot h|<1$.

{\bf Proof of claim}: Assume that $h\in C_{f}(\Omega_n')$. The other case is similar. First note that since $f$ is $1$-periodic and admits fixed points in $\mathbf{R}$, so does $h$ since it is conjugate to $f$. 
Since $\Omega_n'$ is $1$-periodic, first we claim that whenever $I\subset \R$ is a compact interval of length less than $1$ and $\alpha\in \Omega_n'$, then $I\cdot \alpha$ is also an interval of length less than $1$. Indeed, a $1$-periodic homeomorphism maps any interval of length $1$ to an interval of length $1$ by definition, so this follows. 

Let $x\in \R$. Since $f\in F_{2^n}'$, each component $I$ of the support of $f$ has length less than $1$.
The component of the support of $h$ that contains $x$ is conjugate to some component of the support of $f$ by a $1$-periodic homeomorphism, i.e. some element in $\Omega_n'$. By the observation in the previous paragraph, our claim follows.

We now assume by way of contradiction that: $$g=g_1...g_kt_n^m\qquad \text{ for }g_i\in C_f(\Omega_n')\cup C_{f^{-1}}(\Omega_n'),m\in \Z\qquad \text{ so that } k<\left\lfloor{\frac{n}{2}}\right\rfloor.$$
Then $(0\cdot g-0\cdot g_1...g_k)\in n\mathbf{Z}$, since $t_n\in \textup{Homeo}^+(\R)$ is the element $x\cdot t_n=x+n$. 
From the claim above, it follows that for any $x\in \mathbf{R}$ and $1\leq i\leq k$, $|x-x\cdot g_i|<1$.
Then $$0\cdot g\in (\floor{\frac{n}{2}},\floor{\frac{n}{2}}+1)\qquad (0\cdot g-0\cdot g_1...g_k)\in n\mathbf{Z}$$ implies that: $$0\cdot g_1...g_k\in (J+n\mathbf{Z})\qquad \text{ for }J=(\left\lfloor\frac{n}{2}\right\rfloor,\left \lfloor\frac{n}{2}\right\rfloor+1).$$
If $k<\floor{\frac{n}{2}}$, then this cannot be the case since for $x\in \mathbf{R}$ and $1\leq i\leq k$, $|x-x\cdot g_i|<1$.
\end{proof}

A consequence of this is the required lower bound on the Ulam width.

\begin{cor}\label{maincorlb}
$\mathcal{R}(\Gamma_n')\geq \floor{\frac{n}{2}}$.
\end{cor}

\begin{proof}
Let $f,g\in \Omega_n'$ be the elements satisfying the hypothesis of Proposition \ref{mainproplb}.
Recall that we fixed the notation $\phi_n:\Omega_n\to \Gamma_n$ for the map $\Omega_n\to \Omega_n/\Lambda_n$. Let $\alpha=\phi_n(f),\beta=\phi_n(g)$. If $\Gamma_n'$ was $(\floor{\frac{n}{2}}-1)$-uniformly simple, then we can find $\beta_1,...,\beta_k\in C_{\alpha}(\Gamma_n')\cup C_{\alpha^{-1}}(\Gamma_n')$
such that $k\leq \floor{\frac{n}{2}}-1$ and $\beta=\beta_1...\beta_k$.
However, choosing preimages $g_it_n^{m_i}=\phi_n^{-1}(\beta_i)$ for $g_i\in C_f(\Omega_n')\cup C_{f^{-1}}(\Omega_n'), m_i\in \mathbf{Z}$ gives us $g=g_1...g_kt_n^m$ for some $m\in \mathbf{Z}$.
Since $k<\floor{\frac{n}{2}}$, this contradicts the conclusion of Proposition \ref{mainproplb}.
\end{proof}

Another consequence of this is the required lower bound on the commutator width.

\begin{proof}[Proof of Theorem \ref{mainperfect}]
First note that the groups $\{\Gamma_n'\}_{n\geq 2}$ are uniformly perfect since they are uniformly simple.
Assume by way of contradiction that $\Gamma_n'$ is $(\floor{\frac{n}{4}}-1)$-uniformly perfect.
Using the minimality of $\Omega_n'$, we find $f\in \Omega_n'$ such that $0\cdot f\in (\floor{\frac{n}{2}},\floor{\frac{n}{2}}+1)$.
By construction, $f\notin \Lambda_n$, and so $g=\phi_n(f)$ is nontrivial.
From our assumption, it follows that $g$ is a product of $k$ commutators of elements in $\Gamma_n'$ for $k\leq (\floor{\frac{n}{4}}-1)$.
In particular, there exist elements $h_1,l_2,...,h_k,l_k\in \Omega_n'$ such that $f=[h_1,l_1]...[h_k,l_k]t_n^m$ for some $m\in \mathbf{Z}$.
Given $1$-periodic homeomorphisms $g_1,g_2\in \textup{Homeo}^+(\mathbf{R})$, it is easy to see that $h=[g_1,g_2]$ has the property that for each $x\in \mathbf{R}, |x-x\cdot h|<2$.
It follows that for $k\leq (\floor{\frac{n}{4}}-1)$: $$0\cdot [h_1,l_1]...[h_k,l_k]\notin (J+n\mathbf{Z})\qquad J=(\left\lfloor\frac{n}{2}\right\rfloor,\left \lfloor\frac{n}{2}\right\rfloor+1).$$
This contradicts our assumption that $0\cdot f\in (\floor{\frac{n}{2}},\floor{\frac{n}{2}}+1)$.
\end{proof}

\section{The finiteness properties of \texorpdfstring{$\{\Gamma_n'\}_{ n\geq 2}$}.}

In this section we shall prove that $\Gamma_n'$ is of type $F_{\infty}$ for each $n\in \mathbf{N},n\geq 2$. For $I\subseteq [0,1]$ a closed interval and $G\leq \textup{PL}^+[0,1]$, define: $$\textup{Rstab}_G^c(I)=\{f\in \textup{Rstab}_G(I)\mid f'_+(\iinf(I))=f'_-(\ssup(I))=1\}$$
Given a group $G\leq \Omega_n$ and an interval $I\subset \mathbf{R}, |I|\leq 1$, 
define $$\Upsilon_{G}(I)=\textup{RStab}_G(\Int(I)+\mathbf{Z})\qquad \Upsilon_{G}^c(I)=\{f\in \Upsilon_{G}(I)\mid f'_+(\iinf(I))=f'_-(\sup(I))=1\}$$

Let $G$ be a group and $N$ a subgroup such that $N$ is of type $\mathbf{F}_{\infty}$, $G$ is finitely generated, and $G'\leq N$.
Let $H$ satisfy that $N\leq H\leq G$. Then the pair $(N,G)$ is said to form a \emph{casing pair} for $H$. The following technical lemma was proved in \cite{HydeLodhaFPSimple} (Lemma $3.14$).

\begin{lem}\label{casing}
If a group $H$ admits a casing pair $N,G$, then $H$ is of type $\mathbf{F}_{\infty}$.
\end{lem}

Now we shall prove a series of technical results that shall be the ingredients in establishing the finiteness properties of $\Omega_n'$.

\begin{lem}\label{FPlem1}
For each $a,b\in \mathbf{Z}[\frac{1}{2}], 0<b-a\leq 1$, there exists $f\in \Upsilon_{\Omega_n'}([a,b])$ such that $f_+'(a),f_-'(b)=2^n$.
\end{lem}

\begin{proof}
First we consider the case when $b-a<1$. 
From Proposition \ref{periodproximal}, there is an $f_1\in \Omega_n'$ such that $(a,b)\cdot f_1\subset (0,1)$.
Using Proposition \ref{Fntrans}, we find $f_2\in F_{2^n}$ such that $$a\cdot f_1 < (a\cdot f_1)\cdot f_2< b\cdot f_1< (b\cdot f_1)\cdot f_2$$
Next, we construct $h\in F_{2^n}$ with support in $(a\cdot f_1, (b\cdot f_1)\cdot f_2^{-1})+\mathbf{Z}$ whose right slope at $a\cdot f_1$ is $2^n$ and left slope at $(b\cdot f_1)\cdot f_2^{-1}$
is $\frac{1}{2^n}$. Thus $f=f_1(f_2^{-1}h^{-1} f_2h)f_1^{-1}=[f_2,h]^{f_1^{-1}}$ is the required element.
Next, assume the case $b=a+1$. Let $c\in (a,b)\cap \mathbf{Z}[\frac{1}{2}]$ and find such elements $f_1,f_2$ for $[a,c],[c,b]$, respectively, using the above.
It follows that $f=f_1f_2$ is the required element.
\end{proof}

\begin{lem}\label{newing}
For each $a,b\in \mathbf{Z}[\frac{1}{2}], 0<b-a\leq 1$, the group $\Upsilon_{\Omega_n}([a,b])$ is of type $F_{\infty}$ and $\Upsilon_{\Omega_n}([a,b])'=\Upsilon_{\Omega_n}([a,b])''$.
\end{lem}

\begin{proof}
If $b-a<1$, then from Proposition \ref{periodproximal} there is an $f\in \Omega_n'$ such that $I=[a,b]\cdot f\subset (0,1)$.
It follows that $\Upsilon_{\Omega_n}([a,b])\cong \Upsilon_{\Omega_n}(I)=\Upsilon_{F_{2^n}}(I)\cong F_{2^n}$, where the last identification follows from Lemma \ref{Fnproximal}.
Our conclusion follows.
Now assume that $b=a+1$, and without loss of generality assume $b\in (0,1]$.
Let $f\in \Upsilon_{\Omega_n}([a,b])\cap F_{2^n}$ be such that $f_+'(a)=2^n$.
Let $c\in (a,b)\cap \mathbf{Z}[\frac{1}{2}]$ be such that $x\cdot f>x, \forall x\in (a,c]$.
From Lemma \ref{AHNN}, $\Upsilon_{\Omega_n}([a,b])=\langle f, \Upsilon_{\Omega_n}([c,b])\rangle $ is an ascending HNN extension with base group $\Upsilon_{\Omega_n}([c,b])\cong F_{2^n}$ (from the above, since $b-c<1$), hence is of type $\mathbf{F}_{\infty}$ from Proposition \ref{AHNNF}. 

Let $I=[a,b]$. We show that $\Upsilon_{\Omega_n}(I)''=\Upsilon_{\Omega_n}(I)'$ by showing the quotient $$\phi: \Upsilon_{\Omega_n}(I)\to \Upsilon_{\Omega_n}(I)/\Upsilon_{\Omega_n}(I)''$$ is abelian.
Again, without loss of generality assume $b\in (0,1]$,
and choose $k_1\in \mathbf{Z}[\frac{1}{2}]\cap (0,b)$.
Let $f\in \Upsilon_{\Omega_n}([a,k_1])\cap F_{2^n}$ be such that $f_+'(a)=2^n$.
Therefore, $\Upsilon_{\Omega_n}(I)=\langle f,G\rangle$, where $G= \{g\in \Upsilon_{\Omega_n}(I)\mid g_+'(a)=1\}$.
For any $\gamma\in G$ we can find $\gamma_1\in \Upsilon_{\Omega_n}(I)''$ so that $\gamma_1^{-1}\gamma \gamma_1\in \Upsilon_{\Omega_n}([k_1,b])$.
It follows that $\phi(\Upsilon_{\Omega_n}(I))=\phi(\langle f\rangle \oplus \Upsilon_{\Omega_n}([k_1,b]))$.
Since every proper quotient of $\Upsilon_{\Omega_n}([k_1,b])\cong F_{2^n}$ is abelian, $\phi(\Upsilon_{\Omega_n}(I))$ is abelian.
\end{proof}

\begin{lem}\label{FPlem3}
For all $K$ satisfying $\Omega_n'\leq K\leq \Omega_n$ and $I=[a,b]\subset \mathbf{R}, a,b\in \mathbf{Z}[\frac{1}{2}]$ with $|I|\leq 1$, 
$\Upsilon_{K}(I)$ is of type $\mathbf{F}_{\infty}$.
\end{lem}

\begin{proof}
Assume without loss of generality that $b\in (0,1]$.
Using Lemma \ref{FPlem1}, we find $f\in \Upsilon_{\Omega_n'}(I)$ such that $f_+'(a),f_-'(b)>1$.
In particular, $f\in \Upsilon_{K}(I)$.
Using Proposition \ref{periodproximal}, upon replacing $I$ by $I\cdot g$ (and $f$ by $f^g$) for some $g\in \Omega_n'$ if needed, we can assume the following:

\begin{enumerate}
\item If $|I|<1$, then $I\subset (0,1)$. 
\item There exists a closed interval $J\subset (0,1)\cap (a,b)$ with endpoints in $\mathbf{Z}[\frac{1}{2}]$ so that
$J\cdot f\subset J$ and $\bigcup_{n\in \mathbf{N}}J\cdot f^{-n}=(a,b)$.
\end{enumerate}

We will prove the lemma by isolating a group $H\cong F_{2^n}$ such that 
($\langle f,H\rangle, \Upsilon_{\Omega_n}(I)$) is 
a casing pair for $\Upsilon_{K}(I)$, which will imply that $\Upsilon_{K}(I)$ is of type $\mathbf{F}_{\infty}$ using Lemma \ref{casing}.

Let $J_0=J\cdot f$. Note that $J_0\subset int(J)$. We choose: $$s_1,...,s_{2^n}\in \Upsilon_{F_{2^n}'}(J)\leq \Upsilon_{K}(I)$$ such that the intervals defined as: $$\{J_i=J_0\cdot s_i\mid 1\leq i\leq 2^n\}$$ 
are pairwise disjoint intervals in $J\setminus J_0$.
We know that: $$\Upsilon_{\Omega_n}(J_0)\cong \Upsilon_{F_{2^n}}(J_0)\cong F_{2^n}$$ (the latter follows from Lemma \ref{Fnproximal}). 
Fix a generating set $u_1,...,u_{2^n}$ for $\Upsilon_{\Omega_n}(J_0)$ and set $v_i=u_is_i^{-1}u_i^{-1}s_i$. Since the set of relators of $F_{2^n}$ in the prescribed generating set $u_1,...,u_{2^n}$ are all 
products of commutators, $H=\langle v_1,...,v_{2^n}\rangle$ also satisfies them and hence is isomorphic to $F_{2^n}$.
It follows that $H'=\Upsilon_{\Omega_n}(J_0)'$.
Moreover, by definition, $H\leq \Upsilon_{\Omega_n}(J)'\leq \Omega_n'\leq K$, and from the above recall that $f\in \Upsilon_{K}(I)$. So $\langle f, H\rangle\leq \Upsilon_{K}(I)$.

First, note that $f^{-1}Hf\subset H$, since each $f^{-1}v_if\in \Upsilon_{\Omega_n}(J_0)'=H'$.
So by Lemma \ref{AHNN}, $\langle f,H\rangle$ is an ascending HNN extension of $H$, and hence of type $\mathbf{F}_{\infty}$ by Proposition \ref{AHNNF}.
Also, $\Upsilon_{\Omega_n}(I)$ is finitely generated from Lemma \ref{newing}.
It remains to show that $\Upsilon_{\Omega_n}(I)'\leq \langle f, H\rangle$, which reduces to:

{\bf Claim}: $\Upsilon_{\Omega_n}(I)'\subseteq \bigcup_{n\in \mathbf{Z}}f^{n}Hf^{-n}$.

{\bf Case $1$}: $|I|<1$. In this case, we assumed at the beginning that $I\subset (0,1)$.
Therefore, from Lemma \ref{FPlem2}: $$\Upsilon_{\Omega_n}(I)'\cap \Upsilon_{\Omega_n}^c(J_0)=\Upsilon_{\Omega_n}(J_0)'=H'$$ since $\Upsilon_{\Omega_n}(L)=\Upsilon_{F_{2^n}}(L)$ for any interval $L\subset [0,1]$.
Indeed, for any $k\in \Upsilon_{\Omega_n}(I)'$, there is an $n\in \mathbf{N}$ such that $f^{-n}kf^{n}\in \Upsilon_{\Omega_n}^c(J_0)$, and so $f^{-n}kf^{n}\in \Upsilon_{\Omega_n}(J_0)'=H'$.

{\bf Case $2$}: $|I|=1$. 
Let $k\in \Upsilon_{\Omega_n}(I)'$.
We need to show that there is an $n\in \mathbf{Z}$ such that $f^{-n}kf^{n}\in H'$.
From Lemma \ref{newing}, $\Upsilon_{\Omega_n}(I)''=\Upsilon_{\Omega_n}(I)'$, ensuring that each element of $\Upsilon_{\Omega_n}(I)'$ lies in 
$\Upsilon_{\Omega_n}(I_1)'$ for some interval $I_1\subset I, |I_1|<1$.
Upon conjugating by a power of $f$, we may assume that $I_1\subset (0,b)$.
Indeed, since $k\in \Upsilon_{\Omega_n}(I_1)'$, we also have that $k\in \Upsilon_{\Omega_n}(0,b)'$.
Moreover, there is an $n\in \mathbf{N}$ such that $f^{-n}kf^{n}\in \Upsilon_{\Omega_n}^c(J_0)$. Note that it also holds that $f^{-n}kf^{n}\in \Upsilon_{\Omega_n}(0,b)'$. And so from Lemma \ref{FPlem2}: $$f^{-n}kf^{n}\in \Upsilon_{\Omega_n}(0,b)'\cap \Upsilon_{\Omega_n}^c(J_0)=\Upsilon_{\Omega_n}(J_0)'=H'$$
This finishes our proof.
\end{proof}

For each $K$ such that $\Omega_n'\leq K\leq \Omega_n$ and $P\subset \mathbf{Z}[\frac{1}{2}]$, 
define: $$K_P=\{f\in K\mid x\cdot f=x, \forall x\in P\}$$ 
and further define:
$$K_{P,1}=\{f\in K\mid (x+\mathbf{Z})\cdot f=x+\mathbf{Z},\forall x\in P\}$$ 
Note that elements of $K_{P,1}$ are not required to fix $P$ pointwise. Also, note that since $\Omega_n$ is $1$-periodic, we have $K_P=K_{P+\mathbf{Z}}$.

\begin{prop}\label{FPProp1}
For $K$ so that $\Omega_n'\leq K\leq \Omega_n$ and every nonempty finite set $P\subset \mathbf{Z}[\frac{1}{2}]$, $K_P=K_{P+\mathbf{Z}}$ is of type $\mathbf{F}_{\infty}$.
\end{prop}

\begin{proof}
Using the fact that $K_{P}=K_{P+\mathbf{Z}}$, changing $P$ if needed and keeping $K_P$ fixed, suppose that $max(P)<min(P)+1$.
Thus, this provides $|P|$ 
(left closed, right open) intervals $L_1,...,L_{|P|}$ such that $L=\bigcup_{1\leq i\leq |P|}L_i$ is an interval whose closure has length $1$
and $L+\mathbf{Z}$ partitions $\mathbf{R}$.
Let $$R=\prod_{1\leq i\leq |P|}\Upsilon_{\Omega_n}(L_j)\qquad R_1=\prod_{1\leq i\leq |P|}\Upsilon_{K}(L_j)$$
From our hypothesis, it follows that $R'\leq R_1\leq K_P\leq R$.
Now $R_1,R$ are of type $\mathbf{F}_{\infty}$, from Lemma \ref{FPlem3} and Proposition \ref{extensionfiniteness}.
So they form a casing pair for $K_P$, which is henceforth of type $\mathbf{F}_{\infty}$ from Lemma \ref{casing}.
\end{proof}

\begin{thm}\label{FPmain}
Each $K$ satisfying $\Omega_n'\leq K\leq \Omega_n$ is of type $\mathbf{F}_{\infty}$.
\end{thm}

\begin{proof}
We consider the actions of $K, \Omega_n$ on $X=\mathbf{Z}[\frac{1}{2}]/\mathbf{Z}$,
and the natural extension of this action to that on the simplicial complex whose $k$-simplices are the $(k+1)$-element subsets of $X$.
This complex is easily seen to be contractible.
For each nonempty finite set $P\subset X$, the pointwise stabilizer of this simplex is $K_{P,1}$, which is an extension of $K_{P+\mathbf{Z}}$ by a cyclic (trivial or infinite cyclic) group.
Since from Proposition \ref{FPProp1} $K_{P+\mathbf{Z}}$ is of type $\mathbf{F}_{\infty}$, it follows from applying Proposition \ref{extensionfiniteness} that $K_{P,1}$ is of type $\mathbf{F}_{\infty}$.
From Proposition \ref{Fntrans}, for each $k\in \mathbf{N}\setminus \{0\}$, the action of $F_{2^n}'$ on $(\mathbf{Z}[\frac{1}{2}]/\mathbf{Z})^k$ has finitely many orbits. Hence the same holds for $K$ since $F_{2^n}'\leq \Omega_n'\leq K$.
Using Theorem \ref{brownscriterion}, we conclude that $K$ is of type $\mathbf{F}_{\infty}$.
\end{proof}

\begin{cor}\label{FPCor}
For each $n\geq 2$, $\Gamma_n'$ is of type $\mathbf{F}_{\infty}$.
\end{cor}

\begin{proof}
Using Theorem \ref{Finftyquotients}, it suffices to show that $\Omega_n'$ is type $F_{\infty}$, since $\Gamma_n'=\phi_n(\Omega_n')$ and $ker(\phi_n)=\Lambda_n$ is infinite cyclic. This was shown in Theorem \ref{FPmain}.
\end{proof}

\begin{proof}[Proof of Theorem \ref{main}]
For each $n\geq 2$, $\Gamma_n'$ is uniformly simple from Theorem \ref{unisim}, satisfies that $\mathcal{R}(\Gamma_n')\geq \floor{\frac{n}{2}}$ from Corollary \ref{maincorlb}, and is of type $\mathbf{F}_{\infty}$ from Corollary \ref{FPCor}.
\end{proof}

\bibliographystyle{amsalpha}
\bibliography{bib}
\normalsize

\vspace{0.5cm}

James Hyde.

\noindent{\textsc{Department of Mathematics,
Binghamton University.}}

\noindent{\textit{E-mail address:} \texttt{jameshydemaths@gmail.com}}

Yash Lodha.

\noindent{\textsc{Department of Mathematics, Purdue University}}

\noindent{\textit{E-mail address:} \texttt{ylodha@purdue.edu}} 
\end{document}